\documentclass[10 pt,twocolumn,journal]{IEEEtran}
\IEEEoverridecommandlockouts
\usepackage{cite}
\usepackage{filecontents}
\usepackage[dvipsnames,table]{xcolor}
\usepackage{graphicx}
\usepackage{colortbl}
\usepackage{cuted}
\definecolor{Myred}{rgb}{0.5,0,0}
\definecolor{Mygreen}{rgb}{0,0.5,0}
\definecolor{Myblue}{rgb}{0,0,0.5}
\definecolor{MyLightYellow}{rgb}{0.98, 0.93, 0.36}
   
\ifCLASSINFOpdf
\else
\fi
\usepackage{graphicx}
\ifCLASSINFOpdf
\usepackage{graphicx}
\DeclareGraphicsExtensions{.pdf,.jpeg,.png}
\else
\fi
\usepackage{balance}
\usepackage{amsmath,amsfonts,amsthm,amssymb}
\usepackage{bm}
\usepackage{mathtools}
\DeclarePairedDelimiter{\ceil}{\lceil}{\rceil}
\DeclarePairedDelimiter{\floor}{\lfloor}{\rfloor}

\newtheorem{definition}{Definition}
\newtheorem{proposition}{Proposition}
\theoremstyle{plain}
\newtheorem{example}{Example}[] 

\newtheorem{theorem}{Theorem}
\newtheorem{corollary}{Corollary}
\newtheorem{lemma}{Lemma}
\newcommand{\db}{\overset{\Delta}{=}}

\newcommand{\A}{\mathcal{A}}

\renewcommand{\O}{\mathcal{O}}
\newcommand{\I}{\mathcal{I}}
\newcommand{\X}{\mathcal{X}}
\newcommand{\W}{\mathcal{W}}

\newcommand{\R}{\mathcal{R}}
\newcommand{\C}{\mathcal{C}}
\renewcommand{\L}{\mathcal{L}}
\newcommand{\x}{\bm{x}}
\newcommand{\y}{\bm{y}}
\newcommand{\z}{\bm{z}}
\renewcommand{\b}{\bm{\beta}}
\renewcommand{\a}{\bm{a}}
\renewcommand{\c}{\bm{c}}
\newcommand{\p}{\tb{p}}
\newcommand{\pp}{\tb{p'}}
\newcommand{\q}{\tr{q}}
\newcommand{\qq}{\tr{q'}}
\renewcommand{\r}{\tb{\zeta}}
\newcommand{\rr}{\tb{\zeta'}}
\newcommand{\s}{\tr{\eta}}
\renewcommand{\ss}{\tr{\eta'}}
\newcommand{\e}{\mathtt{e}}
\newcommand{\f}{\mathtt{f}}
\renewcommand{\d}{{\delta}}
\newcommand{\n}{\tb{\mathtt{a}}}
\renewcommand{\o}{\tr{\mathtt{b}}}
\newcommand{\oo}{\tr{\mathtt{c}}}
\newcommand{\nn}{\tb{\mathtt{d}}}
\renewcommand{\A}{\mathtt{A}}
\newcommand{\B}{\mathtt{B}}

\newcommand{\tb}{\textcolor{blue}}
\newcommand{\tr}{\textcolor{Myred}}
\newcommand{\tg}{\textcolor{Mygreen}}
\usepackage{verbatim} 

\theoremstyle{definition}

\begin{document}

\title{Characterizing Oscillations in Heterogeneous Populations of Coordinators and Anticoordinators}
\author{{Pouria Ramazi and Mohammad Hossein Roohi}
\thanks{This paper is submitted in part for presentation at, and publication in the proceedings of, the 61th IEEE Conference on Decision and Control, Cancun, Mexico, December 6-9, 2022 \cite{conferenceVersion}.
This paper additionally extends the Introduction section and includes Section \ref{sec_stability} on the stability analysis, Section \ref{sec_revisitingExample2} on revisiting the example from the stability analysis perspective, Section \ref{sec_synchronousUpdates} on the case with synchronous updates, the Conclusion section, Lemma \ref{lem_updateRule2} in Appendix, and the proof of the results.  
M. H. Roohi is with the Department of Electrical and Computer Engineering, University of Alberta, Canada 
{\tt\small roohi@ualberta.ca}.
P. Ramazi is with
the Department of Mathematics and Statistics, Brock University, Canada,
{\tt\small p.ramazi@gmail.com}.}}%
	\maketitle 
	\begin{abstract}
    Oscillations often take place in populations of decision makers that are either a \emph{coordinator}, who takes action only if enough others do so, or an \emph{anticoordinator}, who takes action only if few others do so.
    Populations consisting of exclusively one of these types are known to reach an equilibrium, where every individual is satisfied with her decision.  
    Yet it remains unknown whether oscillations take place in a population consisting of both types, and if they do, what features they share.
    We study a well-mixed population of individuals, which are either a coordinator or anticoordinator, each associated with a possibly unique threshold and initialized with the strategy $\A$ or $\B$. 
    At each time, an agent becomes active to update her strategy based on her threshold: 
    an active coordinator (resp. anticoordinator) updates her strategy to $\A$ (resp. $\B$) if the portion of other agents who have chosen $\A$ exceeds (falls short of) her threshold, and updates to $\B$ (resp. $\A$) otherwise.
    We define the state of the population dynamics as the distribution over the thresholds of those who have chosen $\A$.
    We show that the population can admit several minimally positively invariant sets, where the solution trajectory oscillates. 
    We explicitly characterize a class of positively invariant sets, prove their invariance, and provide a necessary and sufficient condition for their stability. 
    Our results highlight the possibility of non-trivial, complex oscillations in the absence of noise and population structure and shed light on the reported oscillations in nature and human societies.
	\end{abstract}	
\section{Introduction}
Oscillatory phenomena are observed in nature and human societies:
periodic patterns in biology \cite{hess2000periodic},
fluctuations in market values \cite{li2008gdp}, 
unsettlements in social emotions, 
non-fixation in fashion trends, and
twist of power among political parties \cite{lawrence1898oscillations}.
In most such situations, individuals have to choose between one out of two actions and either pick an action only if enough others do so, which are referred to as \emph{coordinating individuals} or simply \emph{coordinators}, or pick an action only if not many others do so, which are referred to as \emph{anticoordinating individuals} or simply \emph{anticoordinators}. 
Followers in technology markets avoid risk by producing common products, whereas innovators perceive their benefits in building up a monopoly by developing new or rare products \cite{helpman1992, collins2015}. 
``Activator'' and ``repressor'' cells in synthetic microbial consortia, respectively increase and decrease gene expression if the transcription is low and high, resulting in a positive and negative feedback loop \cite{chen2015emergent}.
Populations consisting exclusively of one of these two types of individuals are known to eventually reach an equilibrium state where individuals are satisfied with their decisions \cite{Ramazi2016a}. 
Therefore, it is only the coexistence of the two that may explain the observed oscillations.
However, not every coexistence of the two leads to an oscillation, and even if it does, the characteristics of the oscillation is unknown.  

Scholars from various disciplines have studied populations of decision-making individuals \cite{barabasi1999,schlag1998,barreiro2018constrained,cheng2014finite,lee2008,funk2009}. 
\emph{Evolutionary game theory}, a powerful framework to analyze these populations  \cite{nowak2006,govaert2020zero,zhu2016evolutionary,como2020imitation,zhao2016matrix}, models the above two decision-makers as individuals who play games against each other by choosing one out of two strategies $\A$ and $\B$, earn payoffs according to their payoff matrices, and revise their strategies over time according to the \textit{(myopic) best response update rule} \cite{bopardikar2017convergence,ghaderi2014opinion,ramazi2015,le2020heterogeneous}.
The update rule dictates individuals to pick the strategy maximizing their payoffs against their opponents. 
Therefore, if an individual's payoff matrix is in the form of a \emph{coordination game}, where playing the same strategy as that of the opponent results in a higher payoff, she will be a coordinator, and if it is in the form of an \emph{anticoordination game}, where playing the strategy opposite to that of the opponent results in a higher payoff, she will be an anticoordinator.
In a different framework, coordinators (resp. anticoordinators) can be captured by \emph{linear threshold models} where every individual has a threshold and chooses $\A$ only if the number of others who have chosen $\A$ exceeds (resp. falls short of) her threshold \cite{granovetter1978threshold}. 
In either of these models, a realistic population is \emph{heterogeneous} in individuals' thresholds or payoff matrices, rather than \emph{homogeneous}. 
In \cite{adam2012behavior,adam2012threshold}, networks of all coordinators or all anticoordinators who update their strategies synchronously, that is all at a time, are shown to reach a limit cycle of length at most two.
In \cite{Ramazi2016a}, any network of all coordinators or all anticoordinators who update their strategies asynchronously, that is one at a time, is guaranteed to equilibrate.
In \cite{vanelli2020games,arditti2021equilibria}, sufficient conditions for existence of and finite-time convergence to equilibrium in heterogeneous mixed populations of coordinators and anticoordinators are shown.
Although structured populations are often more popular as they model the heterogeneity in neighbors, many real-world populations are well-mixed, where individuals know the total number or ratio of others who have chosen a particular strategy.
In technology markets, for example, firms typically know how many others are producing a particular product, and individuals deciding whether to volunteer for a non-gevernmental organization (NGO) may have access to the total number of existing volunteers. 
Note that the exact number/ratio is not required, and just knowing whether it is greater than their thresholds suffices for them to accordingly update their choices. 
Well-mixed heterogeneous populations of anticoordinators updating asynchronously are known to equilibrate regardless of the initial condition \cite{ramazi2018}\footnote{The population state either equilibrates or fluctuates between two states in the long-run. 
However, the fluctuations are simply because each individual updates based on the total number of $\A$-players in the population, including herself, rather than the total number of $\A$-players other than herself.}.
Well-mixed heterogeneous populations of coordinators updating asynchronously also always reach an equilibrium state, yet it is not necessarily unique, even for the same initial condition \cite{ramaziUnderReview}. 
However, to best of our knowledge, there is no in-depth study on the non-converging and fluctuating behavior of heterogeneous mixed-populations of both coordinators and anticoordinators, despite their inevitable coexistence in many real-world situations \cite{bodine2013}.
This missing piece in the literature is perhaps key in understanding many oscillatory phenomena.

We study a well-mixed heterogeneous population comprising both coordinators and anticoordinators who play either of the strategies $\A$ or $\B$ and update asynchronously over time.
We take the distribution of $\A$-players over the coordinators and anticoordinators with different thresholds as the state of the system.
First, via a numerical example, we show that a single population can possess an equilibrium as well as several non-singleton positively invariant sets where the state undergoes never-ending fluctuations.
Second, for the first time, we explicitly identify a group of positively invariant sets that are either a singleton, resulting in an equilibrium, or non-singleton, where the solution trajectory fluctuates between its states. 
Third, we study the notion of stability for the invariant sets and derive a necessary and sufficient stability condition.
Fourth, investigate the synchronous version of the dynamics.
The results shed light on observed oscillations in populations of binary decision-makers, and serve as a stepping stone towards providing incentives or establishing regulatory polices in technology market, social and cultural learning, and health plans \cite{bodine2013, legare2015, sato2016}, to either settle or modify the oscillations. 

\section{Model}
Consider a well-mixed population of $n\geq 2$ agents who play either of the two strategies $\A$ or $\B$, correspondingly earn payoffs, and over a discrete time sequence $t \in \mathbb{Z}_{\geq 0}$ revise their strategies to maximize their payoffs. 
More specifically, the four possible payoffs for an agent $j\in\{1,\ldots,n\}$ against her opponent are summarized by the payoff matrix 
\begin{equation*}
    \bordermatrix{
	& \A & \B \cr \A & a_j & b_j \cr \B & c_j & d_j
	},\qquad a_j,b_j,c_j,d_j \in \mathbb{R},
	\label{eq:payoffMatrix}
\end{equation*}
where, for example, $b_j$ is the payoff when agent $j$ plays $\A$, and her opponent plays $\B$.
Given agent $j$, let $A_{j}$ (resp. $B_j$) denote the total number of $\A$-playing (resp. $\B$-playing) agents in the remaining of the population, i.e., excluding agent $j$. 
Then the accumulated payoff to agent $j$ against the remaining of the population is $a_j A_j+b_j B_j$ when she plays $\A$ and $c_j A_j+d_j B_j$ otherwise. 
At each time step $t\in\mathbb{Z}_{\geq0}$, an agent $j$ becomes active to update her strategy at time $t+1$ based on the \emph{(myopic) best-response update rule}, dictating that she chooses the strategy that maximizes her payoff against the remaining population.
Denote the strategy of agent $j$ by $s_j\in\{\A,\B\}$. 
Then her strategy at time $t+1$ would be 
\begin{equation*}
	s_j(t+1)=\begin{cases}
	    \A &a_j A_j(t)+b_j B_j(t)\geq c_j A_j(t)+d_j B_j(t)\\
	    \B &a_j A_j(t)+b_j B_j(t) < c_j A_j(t)+d_j B_j(t)\\
	\end{cases}
\end{equation*}
Knowing $B_j = n-A_j-1$, we simplify the expressions as
\begin{equation}
	s_j(t+1)=
	\begin{cases}
	    \A &\sigma_j A_j(t)\geq \gamma_j (n-1)\\
	    \B &\sigma_j A_j(t)< \gamma_j (n-1)\\
	\end{cases},
	\label{eq:model}
\end{equation}
where $\sigma_j \overset{\Delta}{=} a_j-c_j+d_j-b_j$ and $\gamma_j \overset{\Delta}{=} d_j-b_j$.
For $\sigma_j\neq 0$, define the \emph{temper} of agent $j$ as $\tau^j\overset{\Delta}{=} \gamma_j/\sigma_j (n-1)$.
Note that this is different from the \emph{threshold} of agent $j$, defined as $\gamma_j/\sigma_j$, which is often used in the literature \cite{riehl2018survey}.
Now if $\sigma_j>0$, the update rule becomes
\begin{equation}
	s_j(t+1)=
	\begin{cases}
	    \A &A_j(t)\geq \tau^j\\
	    \B &A_j(t)< \tau^j\\
	\end{cases}.
	\label{eq:coor}
\end{equation}
Namely, the agent plays $\A$ only if the number of other $\A$-players does not fall short of her temper.
We refer to these agents as \emph{coordinators}.
If $\sigma_j<0$, the update rule becomes
\begin{equation}
	s_j(t+1)=
	\begin{cases}
	    \A &A_j(t) \leq \tau^j\\
	    \B &A_j(t) > \tau^j\\
	\end{cases}.
	\label{eq:anti}
\end{equation}
Namely, the agent plays $\A$ only if the number of other $\A$-players does not exceed her temper. 
We refer to these agents as \emph{anticoordinators}.
For the case when $\tau^j \not\in [0,n-1]$ or $\sigma_j = 0$, update rule \eqref{eq:model} is equivalent to one of the above two cases.
So we assume that every agent in the population has some temper $\tau^j\in[0,n-1]$ and is either a coordinator or an anticoordinator.
Note that our setup is similar to that in \cite{Ramazi2016a}.

The agents may represent people aiming at establishing an NGO for a social service program. 
The individuals know how many others have volunteered (played $\A$) and accordingly decide to also volunteer (play $\A$) or not (play $\B$). 
Some may consider the NGO a failure or burdensome if the volunteers are few, and hence, would only join if enough others have joined (coordinators).
Others may find it a duty to join even if no one else does, but they also find it unnecessary for too many to join.
Hence, they would join only if less than enough others have joined (anticoordinators).
Another example is the co-existence of \emph{innovators} and
\emph{followers} in the technology market \cite{helpman1992}. 
Companies can either take an existing market from an entrenched technology (play $\A$) or take on competitors with any form of innovation (play $\B$). 
Some companies are conservative in the sense that enough number of other sectors focusing on a technology convinces them to follow (coordinators).
Others tend to build up a monopoly, so they change their strategy to innovate upon the presence of rivals developing the same technology (anticoordinators).

We categorize all anticoordinators (resp. coordinators) with the same temper as the same \emph{type}, and assume that there are all together $b\geq1$ (resp. $b'\geq1$) different types of anticoordinators (resp. coordinators).
We label the anticoordinator (resp. coordinator) types in the descending (resp. ascending) order of their tempers by $1,2,\ldots, b$ (resp. $1,2,\ldots, b'$). 
So we have a total of $b+b'$ types of agents. 
Denote the temper of an anticoordinator (resp. coordinator) of type $i$ by $\tau_i$ (resp. $\tau'_i)$.
Then we have
\begin{equation} \label{temperOrdering}
    \tau_1 > \tau_2 >\ldots > \tau_b, \qquad
    \tau'_{b'} > \tau'_{b'-1} >\ldots > \tau'_1.
\end{equation}
We take the distribution of the $\A$-players among these types as the \textit{state} of the system:
\begin{equation*}
	\x(t) 
	\overset{\Delta}{=} \left(x_1(t),\ldots,x_b(t),x'_{b'}(t),\ldots,x'_1(t)\right),
\end{equation*}
where $x_i$ (resp. $x'_i$) denotes the number of $\A$-playing anticoordinators (resp. coordinators) of type $i$. 
Clearly, $\x(t)$ lies in the state space
\begin{equation*}
	\begin{aligned}
	\mathcal{X} \overset{\Delta}{=} 
	\Big\{ (&x_1,\ldots,x_{b},x'_{b'},\ldots,x'_1)\in     \mathbb{Z}^{b+b'}_{\geq0}\,\Big|\, \\
	&\ x_i \leq n_i \forall i \in\{1,\ldots,b\},  
	x'_i \leq n'_i \forall i\in\{1,\ldots,b'\}\Big\},
	\end{aligned}
\end{equation*}
where $n_i\geq1$ (resp. $n'_i\geq1$) is the number of anticoordinators (resp. coordinators) of type $i$.
Let $A(\x)$ denote the total number of $\A$-players in the population at state $\x$. 
One can write update rules \eqref{eq:coor} and \eqref{eq:anti} with respect to $A(\x(t))$ and the tempers of the types according to Lemma \ref{lem_updateRule} in the appendix.

The update rules \eqref{eq:coor} and \eqref{eq:anti} together with the activation sequence of the agents govern the dynamics of $\x(t)$, called the \textit{population dynamics}.
We do not impose any assumption on the activation sequence, although a realistic sequence might be generated by a random process. 
A \textit{positively invariant set} under the population dynamics is a set $\O\subseteq\X$ such that if $\x(0)\in\O$, then $\x(t)\in\O$ for all $t\in\mathbb{Z}_{\geq0}$ and under any activation sequence. 
Namely, the solution trajectory never leaves the set after entering it. 
Our goal is to identify these sets.
If $\O$ is a singleton, then it consists of a single equilibrium state, where the solution trajectory settles.
However, if $\O$ is a non-singleton and does not include an equilibrium, then the solution trajectory perpetually fluctuates between several states in $\O$, as shown in the following section.


\section{Example and intuition}   \label{sec:example}
\begin{example}     \label{example1}
    Consider a population of $42$ agents, consisting of four anticoordinating and five coordinating types with tempers
    \begin{equation*}
        \scalebox{.99}{\text{$
        (\tau_1,\tau_2,\tau_3,\tau_4,
        \tau'_5,\tau'_4,\tau'_3,\tau'_2,\tau'_1) = 
        (\!\!\underbrace{18,9,8,7}_{\text{anticoordinating}}\!\!,
     \underbrace{25,19,14,10,5}_{\text{coordinating}}).
     $}}
    \end{equation*}
    The distribution of the population over the types is given by
	\begin{equation*}
	 \scalebox{.99}{\text{$
	 (n_1,n_2,n_3,n_4,
        n'_5,n'_4,n'_3,n'_2,n'_1) = 
        (\!\!\!\!\underbrace{4,3,1,3}_{\text{anticoordinating}}\!\!\!, 
     \underbrace{15,1,10,2,3}_{\text{coordinating}}).
     $}}
    \end{equation*}
     Figure~\ref{fig:InvariantSet1} shows the evolution of the state $\x$ for the initial condition $\x(0)=({\color{Myblue}4}, {\color{Mygreen}3}, {\color{Mygreen}0}, {\color{Mygreen}0}, {\color{Myred}0}, {\color{Myred}0}, {\color{Myred}0}, {\color{Mygreen}0}, {\color{Myblue}3})$.  
     All type-$1$ coordinators and anticoordinators have fixed their strategies to $\A$, and all type-5, 4, and 3 coordinators have fixed their strategies to $\B$. 
     The state fluctuates in a positively invariant set, where only type-2, 3, and 4 anticoordinators and type-2 coordinators may switch strategies. 
     The minimum number of $\A$-players in this set, i.e., 7, exceeds the temper of type-1 coordinators plus one, i.e., $\tb{\tau'_1} + 1 = 6$, ensuring that they always play $\A$ in view of Lemma \ref{lem_updateRule2}.
     On the other hand, the maximum number of $\A$-players equaling 12, ensures all type-1 anticoordinators play $\A$ and all type-3, 4, and 5 coordinators play $\B$. 
     \begin{figure}[ht]
		\centering	\includegraphics[width=1\linewidth]{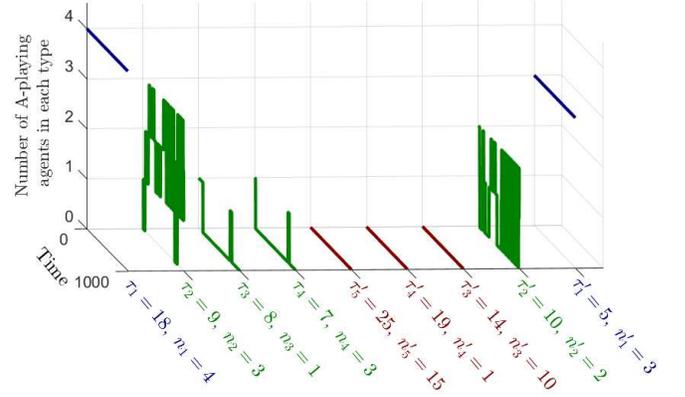}
        \caption{Evolution of agents' decisions in the first positively invariant set.
        Here, $\tb{\p=1}$, $\tr{\q=5}$, $\tr{\qq=3}$, and $\tb{\pp=1}$.
        The $\A$-fixed, $\B$-fixed, and wandering types are indicated by blue, red, and green colors respectively.}
		\label{fig:InvariantSet1}
	\end{figure} 
     
\begin{table}[ht]
  \begin{center}
  \caption{State transition in the first invariant set}
  \begin{tabular}{|r|c|c|}
	\hline 
	State of the system & Active agent  & $\begin{matrix} \text{Number of} \\ \text{$\A$-players}\end{matrix}$ \\  
	\hline 
	\cellcolor{MyLightYellow}
    \!\!$\x(1)=({\color{Myblue}4}, {\color{Mygreen}1}, {\color{Mygreen}1}, {\color{Mygreen}0}, {\color{Myred}0}, {\color{Myred}0}, {\color{Myred}0}, {\color{Mygreen}0}, {\color{Myblue}3})$ &   type-2 anticoordinator & 9\\ 
    $\x(2)=({\color{Myblue}4}, {\color{Mygreen}2}, {\color{Mygreen}1}, {\color{Mygreen}0}, {\color{Myred}0}, {\color{Myred}0}, {\color{Myred}0}, {\color{Mygreen}0}, {\color{Myblue}3})$ &  type-2 coordinator & 10\\ 
    $\x(3)=({\color{Myblue}4}, {\color{Mygreen}2}, {\color{Mygreen}1}, {\color{Mygreen}0}, {\color{Myred}0}, {\color{Myred}0}, {\color{Myred}0}, {\color{Mygreen}1}, {\color{Myblue}3})$ &  type-2 anticoordinator & 11\\ 
    $\x(4)=({\color{Myblue}4}, {\color{Mygreen}1}, {\color{Mygreen}1}, {\color{Mygreen}0}, {\color{Myred}0}, {\color{Myred}0}, {\color{Myred}0}, {\color{Mygreen}1}, {\color{Myblue}3})$ &  type-2 coordinator  & 10\\ 
    \cellcolor{MyLightYellow}$\x(5)=({\color{Myblue}4}, {\color{Mygreen}1}, {\color{Mygreen}1}, {\color{Mygreen}0}, {\color{Myred}0}, {\color{Myred}0}, {\color{Myred}0}, {\color{Mygreen}0}, {\color{Myblue}3})$ &  type-2 anticoordinator  & 9 \\

    $\x(6)=({\color{Myblue}4}, {\color{Mygreen}2}, {\color{Mygreen}1}, {\color{Mygreen}0}, {\color{Myred}0}, {\color{Myred}0}, {\color{Myred}0}, {\color{Mygreen}0}, {\color{Myblue}3})$ &  type-2 coordinator & 10\\ 
    $\x(7)=({\color{Myblue}4}, {\color{Mygreen}2}, {\color{Mygreen}1}, {\color{Mygreen}0}, {\color{Myred}0}, {\color{Myred}0}, {\color{Myred}0}, {\color{Mygreen}1}, {\color{Myblue}3})$ &  type-2 coordinator & 11\\ 
    $\x(8)=({\color{Myblue}4}, {\color{Mygreen}2}, {\color{Mygreen}1}, {\color{Mygreen}0}, {\color{Myred}0}, {\color{Myred}0}, {\color{Myred}0}, {\color{Mygreen}2}, {\color{Myblue}3})$ &  type-2 anticoordinator & 12 \\ 
    $\x(9)=({\color{Myblue}4}, {\color{Mygreen}1}, {\color{Mygreen}1}, {\color{Mygreen}0}, {\color{Myred}0}, {\color{Myred}0}, {\color{Myred}0}, {\color{Mygreen}2}, {\color{Myblue}3})$ &  type-2 anticoordinator& 11\\ 
    $\x(10)=({\color{Myblue}4}, {\color{Mygreen}0}, {\color{Mygreen}1}, {\color{Mygreen}0}, {\color{Myred}0}, {\color{Myred}0}, {\color{Myred}0}, {\color{Mygreen}2}, {\color{Myblue}3})$ &  type-3 anticoordinator& 10\\ 
    $\x(11)=({\color{Myblue}4}, {\color{Mygreen}0}, {\color{Mygreen}0}, {\color{Mygreen}0}, {\color{Myred}0}, {\color{Myred}0}, {\color{Myred}0}, {\color{Mygreen}2}, {\color{Myblue}3})$ &  type-2 coordinator& 9\\ 
    $\x(12)=({\color{Myblue}4}, {\color{Mygreen}0}, {\color{Mygreen}0}, {\color{Mygreen}0}, {\color{Myred}0}, {\color{Myred}0}, {\color{Myred}0}, {\color{Mygreen}1}, {\color{Myblue}3})$ &  type-2 coordinator& 8\\ 
    $\x(13)=({\color{Myblue}4}, {\color{Mygreen}0}, {\color{Mygreen}0}, {\color{Mygreen}0}, {\color{Myred}0}, {\color{Myred}0}, {\color{Myred}0}, {\color{Mygreen}0}, {\color{Myblue}3})$ &  type-4 anticoordinator& 7\\ 
    $\x(14)=({\color{Myblue}4}, {\color{Mygreen}0}, {\color{Mygreen}0}, {\color{Mygreen}1}, {\color{Myred}0}, {\color{Myred}0}, {\color{Myred}0}, {\color{Mygreen}0}, {\color{Myblue}3})$ &  type-2 anticoordinator& 8\\ 
    $\x(15)=({\color{Myblue}4}, {\color{Mygreen}1}, {\color{Mygreen}0}, {\color{Mygreen}1}, {\color{Myred}0}, {\color{Myred}0}, {\color{Myred}0}, {\color{Mygreen}0}, {\color{Myblue}3})$ &  type-4 anticoordinator & 9\\ 
    $\x(16)=({\color{Myblue}4}, {\color{Mygreen}1}, {\color{Mygreen}0}, {\color{Mygreen}0}, {\color{Myred}0}, {\color{Myred}0}, {\color{Myred}0}, {\color{Mygreen}0}, {\color{Myblue}3})$ &  type-3 anticoordinator& 8\\ 
    \cellcolor{MyLightYellow}$\x(17)=({\color{Myblue}4}, {\color{Mygreen}1}, {\color{Mygreen}1}, {\color{Mygreen}0}, {\color{Myred}0}, {\color{Myred}0}, {\color{Myred}0}, {\color{Mygreen}0}, {\color{Myblue}3})$ & & 9\\ 
    \hline
	\end{tabular} \label{tab:state1}
   \end{center}
   \end{table}
   The evolution of $\x$ under a particular activation sequence is shown in Table~\ref{tab:state1}. 
     When all coordinators, except for those who have fixed their strategies to $\A$, are playing $\B$, some low-temper anticoordinators start switching to $\A$.
     This causes some coordinators to switch to $\A$, which in turn makes the aforementioned anticoordinators switch back to $\B$.
     The solution trajectory revisits the states in the positively invariant set, e.g., $\x(1)=\x(5)=\x(17)$.
\end{example}
	
Following the example, we expect the existence of a \emph{benchmark type} $\p\in\{0,1,\ldots,b\}$ of anticoordinators and a benchmark type $\pp\in\{0,1,\ldots,b'\}$ of coordinators such that all type $\tb{1},\tb{2}, \tb{\ldots},\p$ anticoordinators and all type $\tb{1},\tb{2},\tb{\ldots}, \pp$ coordinators eventually fix their strategies to $\A$, and the agents of no other type all fix their strategies to $\A$.
By $\tb{\p = 0}$ (resp. $\tb{\pp = 0}$), we mean that the agents of no anticoordinating (resp. coordinating) type all fix their strategies to $\A$. 
We also expect the existence of two other benchmark types $\q\in\{\p+1,\ldots,b+1\}$ of anticoordinating and $\qq\in\{\pp +1,\ldots,b'+1\}$ of coordinating such that all type $\q,\tr{\q+1},\tr{\ldots},\tr{b}$ anticoordinators and all type $\qq, \tr{\qq+1},\tr{\ldots}, \tr{b'}$ coordinators eventually fix their strategies to $\B$, and the agents of no other type all fix their strategies to $\B$.
By $\tr{\q = b+1}$ (resp. $\tr{\qq = b'+1}$), we mean that the agents of no anticoordinating (resp. coordinating) type all fix their strategies to $\B$. 
We refer to those agents fixing their strategies to $\A$ as \tb{\emph{$\A$-fixed}}, those fixing their strategies to $\B$ as \tr{\emph{$\B$-fixed}}, and the remaining as \tg{\emph{wandering}} agents.
The above long-term behavior is what we anticipate when the solution trajectory enters a positively invariant set.
Namely, for every positively invariant set $\O\subseteq \X$, we expect the existence of benchmarks $(\tb{p},\tr{q},\tr{\qq },\tb{\pp })\in\Omega$, where
\begin{align*}
    \Omega
    \db \Big\{(&\tb{p},\tr{q},\tr{\qq },\tb{\pp })\,\Big|\,  \\
    &\tb{p}\in\{0,1,\ldots,b\}, &&\tr{q}\in\{\tb{p}+1,\ldots,b+1\},\\
    &\tb{\pp }\in\{0,1,\ldots,b'\}, &&\tr{\qq }\in\{\pp +1,\ldots,b'+1\} \Big\},
\end{align*}
such that $\O$ is a subset of the set
\begin{align*}
	&\X_{\tb{p},\tr{q},\tr{\qq },\tb{\pp }}
    \db \Big\{ 
    	\x \in \X\, \big|\, \nonumber \\
        	&\ \ \tb{x_i = n_i} \forall i\in\{\tb{1},\tb{\ldots},\tb{p}\}, 
        	\ \ \tb{x'_i = n'_i} \forall i\in\{\tb{1},\tb{\ldots},\tb{\pp }\}, \nonumber\\
            &\ \ \tr{x_i = 0} \forall i\in\{\tr{q},\tr{\ldots},\tr{b}\}, 
            \ \ \ \ \tr{x'_i = 0} \forall i\in\{\tr{\qq },\tr{\ldots},\tr{b'}\}
            \Big\}. \label{eq:I_1}
\end{align*}

We also expect certain patterns in the wandering agents' strategies.
In particular, starting from the state where all fixed agents have already fixed their strategies and all other agents are playing $\B$, i.e., 
\begin{equation*}
    (\underbrace{\tb{n_1},\tb{\ldots},\tb{n_{\p}},\tg{0},\ldots,\tr{0}}_{\text{anticoordinating}}, 
    \underbrace{\tr{0},\ldots,\tg{0}, \tb{n'_{\pp }}, \tb{\ldots}, \tb{n'_1}}_{\text{coordinating}}),
\end{equation*}
as we update the wandering anticoordinators in the ascending order of their tempers, i.e., from type $\tg{q-1}$ to $\tg{p+1}$, then the total number of $\A$-players never exceeds the temper of the active wandering anticoordinator. 
The intuition behind is that an anticoordinator choosing $\A$ does not make another anticoordinator with a higher temper to switch to $\B$.
We expect this to be also true for every state $\x$ in the invariant set, i.e., for every $\tg{i}\in\{\tg{p+1},\tg{\ldots},\tg{q-1}\}$,
\begin{equation*}
    \sum_{k=1}^{\pp} n'_k +\! \sum_{k=1}^{\p} n_k +\! \sum_{k=i}^{\q-1} x_k \leq \floor{\tg{\tau_{\tg{i}}}}+1,\  \tag{$\mathfrak{L}$}\label{L}
\end{equation*}
Namely, for every type-$i$ wandering anticoordinator, 
the total number of $\A$-playing wandering anticoordinators with a non-less temper, together with the $\A$-fixed agents, does not exceed the type-$i$ temper plus one (see Lemma \ref{lem_updateRule}).

A similar condition should hold if we start from the state where all fixed agents have already fixed their strategies and all wandering agents are playing $\A$, i.e., 
\begin{equation*}
    (\underbrace{\tb{n_1},\ldots,\tg{n_{q-1}},\tr{0},\tr{\ldots},\tr{0}}_{\text{anticoordinating}}, 
    \underbrace{\tr{0},\tr{\ldots},\tr{0}, \tg{n'_{\qq -1}}, \ldots, \tb{n'_1}}_{\text{coordinating}}).
\end{equation*}
Then by updating the wandering anticoordinators in the descending order of their tempers, i.e., from $\tg{p+1}$ to $\tg{q-1}$, the total number of $\A$-players never falls short of the temper of the active wandering anticoordinator. 
The intuition behind is that an anticoordinator choosing $\B$ does not make another anticoordinator with a lower temper to switch to $\A$.
We expect this to be also true for every state $\x$ in the invariant set, i.e., for every $\tg{i}\in\{\tg{p+1},\tg{\ldots},\tg{q-1}\}$,
\begin{equation*}
    \sum_{k=1}^{\qq-1} n'_k +\! \sum_{k=1}^{\p} n_k +\!\!\!\! \sum_{k=\p+1}^{i}\!\!\! x_k + \!\!\!\sum_{k=\tg{i}+1}^{\q-1}\!\! n_k
            \geq \floor{\tg{\tau_i}}+1
             \tag{$\mathfrak{R}$}.\label{R}
\end{equation*}

Given all this, for indices $(\tb{p},\tr{q},\tr{\qq },\tb{\pp })\in\Omega$, define 
\begin{equation*}\scalebox{.98}{\text{$
	\I_{\tb{p},\tr{q},\tr{\qq },\tb{\pp }}
    \db
    \Big\{ \x  \in \X_{\tb{p},\tr{q},\tr{\qq },\tb{\pp }}\, \big|\, 
      \eqref{L}, \eqref{R}\, \forall \tg{i}\in\{\tg{p+1},\tg{\ldots},\tg{q-1}\}\Big\}$}}.
\end{equation*}
For every positively invariant set $\O$, we expect the existence of $(\tb{p},\tr{q},\tr{\qq },\tb{\pp })\in\Omega$ such that $\O\subseteq 	\I_{\tb{p},\tr{q},\tr{\qq },\tb{\pp }}$.
%
The quadruple $(\p,\q,\qq,\pp)$ is not necessarily unique under the population dynamics, and hence, neither is the invariant set. 
Starting from different or even the same initial condition, we may end up at different positively invariant sets, based on the activation sequence. 
Apparently, the population in Example 1 admits another positively invariant set as well as an equilibrium point.
\setcounter{example}{0}
\begin{example}[\textbf{continued}]
    When the state starts from the initial condition $\x(0)=({\color{Mygreen}4}, {\color{Myred}0}, {\color{Myred}0}, {\color{Myred}0}, {\color{Myred}0}, {\color{Mygreen}0}, {\color{Myblue}10}, {\color{Myblue}2}, {\color{Myblue}3})$,
    it wanders in a positively invariant set, different from the previous one (Figure \ref{fig:InvariantSet2}).  
    The set consists of the first 4 states in Table \ref{tab:state2}.
    Interestingly, this time, the state transition in the table is the only possible transition for $\x(t)$. 
    Namely, starting from any initial condition in this set and under any activation sequence, the state follows the same cycle of length 4 to return to that initial condition.
    This behavior is similar to that of a limit cycle with the difference that here, the returning time is not necessarily fixed.
    The reason is that if any agent other than those mentioned in the table become active, they will not switch, and hence, the state remains unchanged. 
\begin{figure}[ht]
	\centering	\includegraphics[width=1\linewidth]{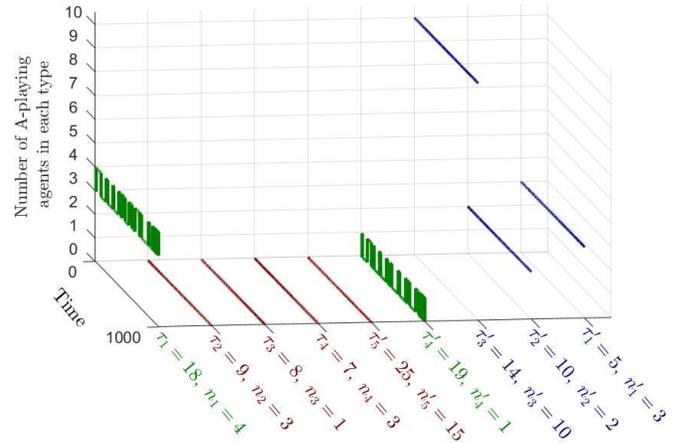}
    \caption{Evolution of agents' decisions in the second positively invariant set with $\tb{\p=0}$, $\tr{\q=2}$, $\tr{\qq=5}$, and $\tb{\pp=3}$.
    The $\A$-fixed, $\B$-fixed, and wandering types are indicated by blue, red, and green colors respectively.}
	\label{fig:InvariantSet2}
\end{figure}
\begin{table}[ht]
  \begin{center}
  \caption{State transition in the second invariant set}
    \begin{tabular}{|c|c|c|}
		\hline 
		State of the system & Active agent & $\begin{matrix} \text{Number of} \\ \text{$\A$-players}\end{matrix}$ \\ 
		\hline 
		\cellcolor{MyLightYellow}$\x(0)=({\color{Mygreen}4}, {\color{Myred}0}, {\color{Myred}0}, {\color{Myred}0}, {\color{Myred}0}, {\color{Mygreen}0}, {\color{Myblue}10}, {\color{Myblue}2}, {\color{Myblue}3})$ &   type-$4$ coordinator & 19 \\ 
        $\x(1)=({\color{Mygreen}4}, {\color{Myred}0}, {\color{Myred}0}, {\color{Myred}0}, {\color{Myred}0}, {\color{Mygreen}1}, {\color{Myblue}10}, {\color{Myblue}2}, {\color{Myblue}3})$ &  type-$1$ anticoordinator & 20 \\ 
        $\x(2)=({\color{Mygreen}3}, {\color{Myred}0}, {\color{Myred}0}, {\color{Myred}0}, {\color{Myred}0}, {\color{Mygreen}1}, {\color{Myblue}10}, {\color{Myblue}2}, {\color{Myblue}3})$ &  type-$4$ coordinator & 19\\
        $\x(3)=({\color{Mygreen}3}, {\color{Myred}0}, {\color{Myred}0}, {\color{Myred}0}, {\color{Myred}0}, {\color{Mygreen}0}, {\color{Myblue}10}, {\color{Myblue}2}, {\color{Myblue}3})$ &  type-$1$ anticoordinator & 18\\ 
        \cellcolor{MyLightYellow}$\x(4)=({\color{Mygreen}4}, {\color{Myred}0}, {\color{Myred}0}, {\color{Myred}0}, {\color{Myred}0}, {\color{Mygreen}0}, {\color{Myblue}10}, {\color{Myblue}2}, {\color{Myblue}3})$ & & 19 \\
		\hline 
	\end{tabular} \label{tab:state2}
   \end{center}
\end{table}

Furthermore, if all coordinators play $\A$, they are enough to keep themselves doing so and stimulate all anticoordinators to play $\B$. 
Thus, the population also possesses the equilibrium point $\x^*=(\tr{0},\tr{0},\tr{0},\tr{0},\tb{15},\tb{1},\tb{10},\tb{2},\tb{3})$. 
\end{example}

Identifying the benchmarks $\p$, $\q$, $\qq$, and $\pp$ is not straightforward. 
The number of $\A$-fixed anticoordinators depends on the number of $\A$-fixed coordinators and the maximum number of $\A$-playing wandering agents.
The number of $\A$-fixed coordinators depends on the number of $\A$-fixed anticoordinators and the minimum number of $\A$-playing wandering agents.
The number of wandering coordinators depends on the number of $\A$-fixed agents and the maximum and minimum number of $\A$-playing wandering anticoordinators, 
which in turn depend on the number of $\A$-fixed agents and wandering coordinators.
This results in several loops, complexifying the identification of the benchmarks as well as proving the invariance of the resulting set.
In what follows, we show how to break the loops to find the benchmarks.
This allows us to characterize a collection of positively invariant sets under the population dynamics. 
We later show in Section \ref{sec_revisiting} that this collection includes the equilibrium in Example \ref{example1}.
The collection also includes two positively invariant sets $\I_{\tb{1}, \tr{5}, \tr{3}, \tb{1}}$ and $\I_{\tb{0}, \tr{2}, \tr{5}, \tb{3}}$, that match precisely the first and second invariant sets in the example.



\section{Positively invariant sets} \label{sec_positivelyInvariantSets}
We provide a sufficient condition for the set $\I_{\p,\q,\qq,\pp}$ to be positively invariant.
More specifically, we introduce quadruples of benchmarks $(\r, \s, \ss, \rr)\in\Omega$ such that $\I_{\r, \s, \ss, \rr}$ is positively invariant. 
The benchmarks depend on the maximum and minimum number of $\A$-players that in the invariant set. 
In Subsection \ref{sec_activationSequences}, we introduce activation sequences, under which the population can reach these extremum numbers of $\A$-players.
Then in Subsection \ref{identifyingPositivelyInvarianSets}, we introduce a collection of pairs $(\r, \d)$, each results in a quadruple $(\r, \s, \ss, \rr)$.
Finally, in Subsection \ref{sec_provingInvariance}, we state the invariance of $\I_{\r, \s, \ss, \rr}$ in Theorem \ref{th1} and proceed to the proof.

Motivated by the structure of the invariant set in Example 1, given a state $\x\in\X$, we define the types $x^{\n},x^{\o} ,x^{\oo},x^{\nn}$ by
\begin{align*}
    &(x^{\n},x^{\o} ,x^{\oo},x^{\nn})\overset{\Delta}{=} \\
    &\underset{(\p,\q,\qq,\pp)\in\Omega}{\arg}\min(-\p+\q+\qq-\pp) \quad \text{ s.t. }\quad \X_{\p,\q,\pp,\qq} \supseteq \x.
\end{align*}
Namely, $\X_{x^{\n},x^{\o},x^{\oo},x^{\nn}}$ is the smallest set in the form of $\X_{\p,\q,\qq,\pp}$ that contains $\x$. 
The special cases $\tb{p=0}$, $\tb{p' =0}$, $\tr{q = b+1}$, and $\tr{q'  = b'+1}$ are treated as in the previous section.
In the absence of the special cases, $\x$ takes the general form
\begin{align*}
    \x = 
    \ &(\underbrace{\tb{n_1},\tb{\ldots},\tb{n_r}, x_{r+1},\ldots, x_{s-1}, \tr{0},\tr{\ldots}, \tr{0}}_{\text{anticoordinating}}, \\
    &\qquad\qquad\underbrace{\tr{0}, \tr{\ldots},\tr{0}, x'_{s'-1},\ldots, x'_{r'+1}, \tb{n'_{r'}},\tb{\ldots},\tb{n'_1}}_{\text{coordinating}}),
\end{align*}
where $x_{r+1}<n_{r+1}$, $x'_{r'+1}<n'_{r'+1}$, $x_{s-1}>0$,  $x'_{s'-1}>0$, and 
 $\tb{r} = x^{\n}, \tr{s} = x^{\o}, \tr{s'} = x^{\oo}, \tb{r'} = x^{\nn}$.

\subsection{Left-to-right and right-to-left activation sequences}   \label{sec_activationSequences}
The key property of coordinators is their conjoint switching to the same strategy: if a coordinator tends to switch to $\A$ (resp. $\B$), so do all other coordinators with a lower (resp. higher) temper.
Consequently, starting from any initial condition, if all coordinators become active in the ascending order of their tempers, we will reach a state with a benchmark temper such that all coordinators with lower tempers play $\A$ and all coordinators with higher tempers play $\B$.  
Formally, let $\overleftarrow{\c}: \X \to \X$ be the function that maps the state $\y\in\X$ to the resulting state after we consecutively update first all type-$1$, next all type-$2$, $\ldots$, and finally all type-$b'$ coordinators.
That is, $\overleftarrow{\c}(\y) = \x(t=\sum_{k = 1}^{b'}n'_k)$ when we start from $\x(0) = \y$ and follow the above activation sequence, which we refer to as the \emph{coordinating right-to-left activation (sequence)}.
The state $\overleftarrow{\c}(\y)$ can be easily shown to take the structure
\begin{equation*}
    \overleftarrow{\c}(\y) = 
    (\!\!\underbrace{*,\ldots,*}_{\text{anticoordinating}}\!\!, 
    \underbrace{0,\ldots,0, n'_i,n'_{i-1} \ldots, n'_1}_{\text{coordinating}}),
\end{equation*}
where $i = \overleftarrow{c}^{\nn}(\y) = \overleftarrow{c}^{\oo}(\y)-1$.
A similar behavior is seen when the coordinators become active in the descending order of their tempers. 
Correspondingly, we define the function $\overrightarrow{\c}$ similar to $\overleftarrow{\c}$ but when first all type-$b'$, next all type-$(b'-1)$, $\ldots$, and finally all type-$1$ coordinators become active, which we refer to as the  \emph{coordinating left-to-right activation}. 

The anticoordinators, however, do not exhibit such a conjoint switching to the same strategy. 
Let $\overleftarrow{\a}: \X \to \X$ be the function that maps a state $\y\in\X$
to the resulting state after we consecutively update first all type-$b$, next all type-$(b-1)$, \ldots, and finally all type-$1$ anticoordinators.
Namely, $\overleftarrow{\a}(\y) = \x(t=\sum_{k = 1}^{b}n_k)$ when we start from $\x(0) = \y$ and follow this activation sequence, which we refer to as the \emph{anticoordinating right-to-left activation}.
Clearly, $\overleftarrow{a}_i(\y)$ then denotes the number of $\A$-playing type-$i$ anticoordinators at the final state.
For $i\in\{1,\ldots, b\}$, let $\overleftarrow{A}^{i}: \X\to\mathbb{Z}_{\geq0}$ with $\overleftarrow{A}^{i}(\y) = A(\x(\sum_{k = i}^{b}n_k))$ given $\x(0) = \y$ and under the above activation sequence; that is, the total number of $\A$-players after we consecutively update all type-$b$, type-$(b-1)$, \ldots, type-$i$, and none of the remaining anticoordinators in $\y$. 
The following result is straightforward.
\begin{lemma}   \label{lem01}
    Let $\x\in\X$.
    Under the anticoordinating right-to-left activation, 
    for each $i \in \{1,\ldots, \overleftarrow{a}^{\o}(\x)-1\}$, one or both of the followings hold:
    \begin{equation*}
        \overleftarrow{a}_i(\x) = n_i \qquad\text{or}\qquad
        \overleftarrow{A}^{i}(\x) = \floor{\tau_i}+1. 
    \end{equation*} 
\end{lemma}
We define the function $\overrightarrow{\a}$ similar to $\overleftarrow{\a}$, but when first all type-$1$, next all type-$2$, $\ldots$, and finally all type-$b$ anticoordinators become active, which we refer to as the \emph{anticoordinating left-to-right activation}. 
The function $\overrightarrow{A}^{i}, i\in\{1,\ldots, b\},$ is defined correspondingly, and the following result is straightforward. 
\begin{lemma}   \label{lem02}
    Let $\x\in\X$.
    Under the anticoordinating left-to-right activation, 
    for each $i \in \{\overrightarrow{a}^{\n}(\y)+1,\ldots, b\}$, one or both of the followings hold:
    \begin{equation*}
        \overrightarrow{a}_i(\x) = 0 \qquad\text{or}\qquad
        \overrightarrow{A}^{i}(\x) = \floor{\tau_i}+1.
    \end{equation*}
\end{lemma}
It proves useful to also define the function $\overrightarrow{\a}^{i}: \X\to\X, i\in\{1,\ldots,b\},$ similar to $\overrightarrow{\a}$ but when we update only type-$i$, type-$(i+1)$, \ldots, type-$b$ anticoordinators in the ascending order, which we refer to as the \emph{anticoordinating left-to-right activation starting from type $i$}. 

\subsection{Identifying positively invariant sets}  \label{identifyingPositivelyInvarianSets}
%
We define the indices $\r$, $\s$, $\ss$, and $\rr$ and provide the intuition behind the definitions, but note that the indices are rigorously defined and do not depend on what we claim to be intuitively true. 
As mentioned in Section \ref{sec:example}, the indices $\r$, $\s$, $\ss$, and $\rr$ are related recursively. 
This hinders any of the indices to be defined independently from the others.
The difficulty is due to the fact that index $\r$, for example, is the \textbf{greatest} $\A$-fixed anticoordinating type. 
If, instead, we search over every arbitrary (yet possibly close to the greatest) $\A$-fixed anticoordinating type $\r$, it no longer has to depend on the other benchmarks, or may only depend on a single benchmark. 
More specifically, given $\r\in\{0,\ldots,b\}$ and $\d\in\{0,\ldots,b'\}$, define the state
\begin{equation*}
    \y^{\r,\d} \db
    (\underbrace{n_1,\ldots,n_{\r},0,\ldots,0}_{\text{anticoordinating}}, 
    \underbrace{0,\ldots,0, n'_{\d}, \ldots, n'_{1}}_{\text{coordinating}}).
\end{equation*}
We approximate $\rr$ by $\delta$ and later tighten it up. 
Representing $\A$-fixed agents, all type-$1$, ..., type-$\r$ anticoordinators must tend to play $\A$ at the state $\y^{\r,\d}$, because the total number of $\A$-players in the invariant set is always non-less than the number of $\A$-players at $\y^{\r,\d}$.
Now, although restrictive, we also force all type-$1$, ..., type-$\d$ coordinators to tend to play $\A$ at $\y^{\r,\d}$, a cost we pay to obtain the approximation.
Therefore, given $\r\in\{0,\ldots,b\}$, we define the set of ``acceptable'' $\delta$'s by 
\begin{equation*}
    \Delta(\r) \db
    \Big\{\d\in\{0,\ldots,b'\}\,\big|\, 
    \tau'_{\d}+1 \leq
        \sum_{i=1}^{\r} n_i + \sum_{i=1}^{\d} n'_i 
            \leq \tau_{\r}+1 \Big\},
\end{equation*}
where we define $\tau'_0=-2$ and $\tau_{0}=n$.
The $\A$-fixed anticoordinators should be resistant to not only the minimum number of $\A$-players in the invariant set, but also the maximum. 
The maximum number of $\A$-players that the solution trajectory can reach from $\y^{\r,\d}$ is obtained by applying
an anticoordinating right-to-left activation to reach $\overleftarrow{\a}(\y^{\r,\d})$, followed by a coordinating right-to-left activation to reach $ \overleftarrow{\c}(\overleftarrow{\a}(\y^{\r,\d}))$.
We then expect all type-$1$, ..., type-$\r$ anticoordinators tend to play $\A$, implying $A(\overleftarrow{\c}(\overleftarrow{\a}(\y^{\r,\d}))) \leq {\tau_{\r}}+1$ in view of Lemma \ref{lem_updateRule}.
Moreover, we want $\r$ to be the maximum type that satisfies the inequality, implying $\tau_{\r+1} < A(\overleftarrow{\c}(\overleftarrow{\a}(\y^{\r,\d})))$.
Thus, we end up at the following ``acceptable'' pairs of $(\r,\d)$:
\begin{align}
    \Psi \db 
    \Big\{(\r,\d) \,|\, &\r\in\{0,\ldots,b\}, \d\in\Delta(\r), \nonumber\\
    & {\tau_{\r+1}} < A(\overleftarrow{\c}(\overleftarrow{\a}(\y^{\r,\d}))) \leq {\tau_{\r}} + 1\Big\}. \label{th1_def1}
\end{align}

Consider the case when $\Psi\neq \emptyset$, and let $(\r,\d)\in \Psi$.
Intuitively, the $\B$-fixed anticoordinators, are those who did not switch their strategies to $\B$ at the state $\overleftarrow{\a}(\y^{\r,\d})$ under the above procedure.
We, therefore, define the benchmark $\s$ by
\begin{equation}
    \s \db 
    \overleftarrow{a}^{\o}(\y^{\r,\d}). \label{th1_def2}
\end{equation}
We drop the argument of the functions when it is clear from the context, e.g., we use $\s$ instead of $\s(\r,\d)$.
We later show that the number of $\A$-playing anticoordinators in the invariant set does not exceed that in $\hat{\y}^{\r,\d} \db \overleftarrow{\a}(\y^{\r,\d})$, which is in the form of
\begin{equation*}\scalebox{.96}{\text{$
    (\underbrace{\overbrace{n_1,\ldots,n_{\r},*,\ldots,*}^{\s}, 0,\ldots,0}_{\text{anticoordinating}}, 
    \underbrace{0,\ldots,0, n'_{\d}, \ldots, n'_{1}}_{\text{coordinating}})$}}.
\end{equation*}
Consequently, the maximum number of $\A$-playing coordinators appears in $\tilde{\y}^{\r,\d} \db \overleftarrow{\c}(\overleftarrow{\a}(\y^{\r,\d}))$. 
Hence, we define
\begin{equation}
    \ss \db 
    \overleftarrow{c}^{\oo}(\overleftarrow{\a}(\y^{\r,\d})). \label{th1_def3}
\end{equation}
The state $\tilde{\y}^{\r,\d}$ will, therefore, take the form
\begin{equation*}\scalebox{.96}{\text{$
    (\underbrace{\overbrace{n_1,\ldots,n_{\r},*,\ldots,*}^{\s}, 0,\ldots,0}_{\text{anticoordinating}}, 
    \underbrace{0,\ldots,0, n'_{\ss-1}, \ldots, n'_{1}}_{\text{coordinating}})$}}.
\end{equation*}
To find the other coordinating benchmark type $\rr$, we consider the state where all wandering agents are playing $\A$:
\begin{equation*}
\z^{\r,\d} \db 
    (\underbrace{n_1,\ldots,n_{\s-1},0,\ldots,0}_{\text{anticoordinating}}, 
    \underbrace{0,\ldots,0, n'_{\ss-1}, \ldots, n'_{1}}_{\text{coordinating}}),
\end{equation*}
and perform an anticoordinating left-to-right activation starting from type $\r+1$ to
to reach $\hat{\z}^{\r,\d} \db \overrightarrow{\a}^{\r+1}(\z^{\r,\delta})$, which can be shown to take the structure
\begin{equation*}\scalebox{.95}{\text{$
    (\underbrace{\overbrace{n_1,\ldots,n_{\r},*,\ldots,*}^{\s},
    0,\ldots,0}_{\text{anticoordinating}}, 
    \underbrace{0,\ldots,0, n'_{\ss-1}, \ldots, n'_{1}}_{\text{coordinating}})$}}.
\end{equation*}
The $\A$-fixed coordinators will not switch to $\B$ even when the population reaches its minimum number of $\A$-players.
On the other hand, we later show that the number of $\A$-playing anticoordinators at the invariant set does not fall short of that at $\hat{\z}^{\r,\d}$.
Hence, the minimum number of $\A$-playing coordinators is obtained by performing a coordinating left-to-right activation to reach $\tilde{\z}^{\r,\d} \db \overrightarrow{\c}(\hat{\z}^{\r,\d})$, which takes the following form:
\begin{equation*}\scalebox{.96}{\text{$
    (\underbrace{\overbrace{n_1,\ldots,n_{\r},*,\ldots,*}^{\s},
    0,\ldots,0}_{\text{anticoordinating}}, 
    \underbrace{0,\ldots,0, n'_{\rr}, \ldots, n'_{1}}_{\text{coordinating}})
$}},
\end{equation*}
where the final benchmark $\rr$ is defined by
\begin{equation}
    \rr \db 
    \overrightarrow{c}^{\nn}(\overrightarrow{\a}^{\r+1}(\z^{\r,\d})).  \label{th1_def4}
\end{equation}

Since the indices $\s,\rr,\ss$ depend merely on $\r$ and $\d$, so does the set $\I_{\r, \s, \ss, \rr}$. 
Hence, to simplify the notations, we define
\begin{equation}
    \I_{\r,\d} \db \I_{\r, \s, \ss, \rr}. \label{th1_def5}
\end{equation}

\subsection{Proving invariance} \label{sec_provingInvariance}
The following is the main result of the paper.
Recall that $\Psi$ is defined in \eqref{th1_def1}, based on which $\s$, $\ss$, and $\rr$ are defined in \eqref{th1_def2} to \eqref{th1_def4}. 
Correspondingly, the set $\I_{\r, \s, \ss, \rr}$, and hence, $\I_{\r,\d}$ is defined in \eqref{th1_def5}.
\begin{theorem} \label{th1}
    Given $(\r,\d)\in \Psi$,
    the set $\I_{\r,\d}$ is positively invariant.
\end{theorem}
In what follows, we provide the necessary lemmas to prove this theorem.
For simplicity, we drop the superscripts ${\r,\d}$ from the states, e.g., we use $\y$ instead of $\y^{\r,\d}$.
The main part of the proof is to show that given $\x(0)\in\I_{\r,\d}$,
\begin{equation*}
    x^{\n}(t) \geq \r, 
    x^{\o}(t) \leq \s,
    x^{\oo}(t) \leq \ss,
    x^{\nn}(t) \geq \rr\  \forall t\geq 0,
\end{equation*}
which we do by Lemmas \ref{lem1}, \ref{lem10}, \ref{lem11} and \ref{lem7}.
We start by showing $x^{\n}(t) \geq \r$, that is, all type-$1$, \ldots, type-$\r$ anticoordinators will keep playing $\A$, given that we start from the positively invariant set. 
The idea is to use induction and show that as long as $x^{\n}(t)\geq \r$, the number of $\A$-players does not exceed that in $\tilde{\y}$ as stated in Lemma \ref{lem9}, which in turn ensures $x^{\n}(t+1)\geq \r$. 
First, we show the following result.
In what follows, we assume $(\r,\d)\in \Psi$.
\begin{lemma}   \label{lem15} 
    Consider $\x(0)\in\I_{\r,\d}$.
    If for some $T\in\mathbb{Z}_{\geq0}$,  
    \begin{equation}   \label{lem15-eq01}
        x^{\n}(t) \geq \r \quad \forall t\in\{0,\ldots, T\},
    \end{equation}
    then  
    \begin{equation}   \label{lem15-eq02}
         x^{\nn}(t) \geq \d \quad \forall t\in\{0,\ldots, T\}.
    \end{equation}
\end{lemma}
\begin{proof}
    First, we show by contradiction that $\rr \geq \d$.
    Assume on the contrary, $\rr < \d$.
    By definition, $\rr = \overrightarrow{c}^{\nn}(\hat{\z})$.
    Hence, starting from the state $\hat{\z}$, under the coordinating left-to-right activation and after all type-$b'$, ..., type-$(\delta+1)$ coordinators are activated,
    the total number of $\A$-players becomes insufficient for the $\delta$-type coordinators to keep playing $\A$. 
    Therefore,
    \begin{equation*}
        \sum_{i = 1}^{b} \hat{z}_i + \sum_{i=1}^{\d} n'_i
        \leq \tau'_{\d}.
    \end{equation*}
    On the other hand, $ \sum_{i = 1}^{b} \hat{z}_i \geq  \sum_{i = 1}^{b} y_i$, implying  
    \begin{equation*}
        \sum_{i = 1}^{b} y_i + \sum_{i=1}^{\d} n'_i
        \leq \tau'_{\d},
    \end{equation*}  
    which contradicts the definition of $\d$, yielding $\rr \geq \d$.
    On the other hand, $\x(0)\in\I_{\r,\d}$ yields $x^{\nn}(0) \geq \rr$.
    Hence, $x^{\nn}(0) \geq \d$. 
    We complete the proof by induction on $t$. 
    The inequality in \eqref{lem15-eq02} holds for $t = 0$. 
    Suppose it holds for $t = k, k<T$.
    Suppose a type-$p, p\in\{1,\ldots, \d\},$ coordinating agent is active at $t=k$; otherwise, the result is trivial. 
    In view of \eqref{lem15-eq01}, and since the inequality in \eqref{lem15-eq02} holds for $t = k$, it holds that
    \begin{equation*}
        A(\x(k))
        \geq
        \sum_{i=1}^{\r} x_i(k) + \sum_{i=1}^{\d} x'_i(k)
        = \sum_{i=1}^{\r} n_i + \sum_{i=1}^{\d} n'_i 
        \geq \tau'_{\d} + 1.
    \end{equation*}
    Hence, the active agent keeps playing $\A$ at $t=k+1$, completing the proof.
\end{proof}
Let $M$ denote the total number of $\A$-playing coordinators at the state $\y$, i.e., 
$
    M \db \sum_{i = 1}^{\d}n'_i.
$
\begin{lemma}   \label{lem2}
    Consider $\x(0)\in\I_{\r,\d}$.
    If for some $T\in\mathbb{Z}_{\geq0}$,  
    \begin{equation}   \label{lem2-eq01}
        x^{\n}(t) \geq \r \quad \forall t\in\{0,\ldots, T\},
    \end{equation}
    then  
    \begin{equation*}
         \sum_{i=1}^b x_i(T) \leq \sum_{i=1}^b\hat{y}_i. 
    \end{equation*}
\end{lemma}
\begin{proof}
    First, note that in view of Lemma \ref{lem15}, \eqref{lem2-eq01} implies
    \begin{equation}    \label{lem2-eq02}
        \sum_{i = 1}^{b'}x'_i(t) 
        \geq M \quad \forall t\in\{0,\ldots, T\}.
    \end{equation}
    Next, by definition, $\hat{y}^{\n} \leq b$,
    $\hat{y}^{\o} \geq \hat{y}^{\n}+1$, and $\s=\hat{y}^{\o}$. Hence,
    exactly one of the following cases holds:
    
    \emph{Case 1:} $\hat{y}^{\n} = b$.
    Hence, $\hat{\y}$ already contains the maximum number of $\A$-playing anticoordinators in the population, making the result trivial.
    
    \emph{Case 2:} $\hat{y}^{\n} < b$ and $\hat{y}^{\o}= \hat{y}^{\n}+1$.
    Since none of the type-$\s$ anticoordinators switch to $\A$ upon the anticoordinating right-to-left activation on the state $\y$, 
    in view of Lemma \ref{lem_updateRule}, 
    \begin{equation}    \label{lem2-eq2}
        \sum_{i=1}^{\r} n_i + M > {\tau_{\s}}.
    \end{equation}
    On the other hand, in view of \eqref{lem2-eq01} and \eqref{lem2-eq02}, 
    \begin{equation}    \label{lem2-eq3}
        A(\x(t)) 
         \geq \sum_{i=1}^{\r} n_i + M 
         \overset{\eqref{lem2-eq2}}{>} {\tau_{\s}} 
         \quad \forall t\in\{0,\ldots,T\}. 
    \end{equation}
    Now since
    $
        x_i(0) = 0  \forall i\in\{\s,\ldots, b\}
    $,
    the following holds due to \eqref{lem2-eq3} and Lemma \ref{lem_updateRule}:
    $
        x_i(T)  = 0  \forall i\in\{\s,\ldots, b\}.
    $
    Therefore, the proof is complete for this part as
    \begin{equation*}
         \sum_{i=1}^b x_i(T) 
         \leq \sum_{i=1}^{\s-1} n_i
         = \sum_{i=1}^b\hat{y}_i,
    \end{equation*}
    where the last equality is due to $\s= \hat{y}^{\o}= \hat{y}^{\n}+1$.
    
    \emph{Case 3:} $\hat{y}^{\n} < b$ and $\hat{y}^{\o}>\hat{y}^{\n}+1$. 
    Let $\hat{y}^{\n} = r$.
    The inclusion $(\r,\d)\in\Psi$ results in $r\geq \r$.
    Thus, 
    \begin{equation*}
        A(\hat{\y}) 
        = \overleftarrow{A}^{r+1}(\y)  + \sum_{i = \r+1}^{r} n_i.
    \end{equation*}
    On the other hand, $r+1\in\{1,\ldots,\hat{y}^{\o}-1\}$, and $\hat{y}_{r+1}\neq n_{r+1}$.
    Thus, in view of Lemma \ref{lem01},
    \begin{gather}
        A(\hat{\y}) 
        = \floor{\tau_{r+1}}+1 + \sum_{i = \r+1}^{r} n_i.\nonumber\\
        \Rightarrow
        \sum_{i=1}^b \hat{y}_i
        = \floor{\tau_{r+1}}+1 + \sum_{i = \r+1}^{r} n_i - M.  \label{lem2-eq4}
    \end{gather}   
    Because of \eqref{lem2-eq02}, the result is trivial if $A(\x(T)) \leq \floor{\tau_{r+1}}+1$.
    So consider the case with $A(\x(T)) > \floor{\tau_{r+1}} + 1$.
    Then either of the followings happens:
    
    \emph{Case 3.1:} $\exists s\in\{0,\ldots,T-1\}:$ 
    \begin{equation} \label{lem2-eq5}
        A(\x(s)) = \floor{\tau_{r+1}}+1,
    \end{equation}
    \begin{equation*} 
        A(\x(t)) > \floor{\tau_{r+1}}+1\quad \forall t\in\{s+1, \ldots, T\}.
    \end{equation*}
    In view of \eqref{lem2-eq5} and \eqref{lem2-eq02}, 
    \begin{equation} \label{lem2-eq15}
        \sum_{i=1}^b x_i(s) \leq \floor{\tau_{r+1}}+1 - M.
    \end{equation}
    On the other hand, none of the type-$i,i\geq r+1,$ anticoordinators switch to $\A$ after time $s$. 
    Moreover, all type-$i,i\leq \r,$ anticoordinators play $\A$ for all time $t\in\{0,\ldots, T\}$. 
    So compared to $\x(s)$, at most $\sum_{i=\r+1}^{r}n_i$ more anticoordinators may play $\A$ at $\x(T)$.
    Hence, according to \eqref{lem2-eq15},
    \begin{equation*}
        \sum_{i=1}^b x_i(T)
        \leq \floor{\tau_{r+1}} + 1+ \sum_{i = \r+1}^{r} n_i - M,
    \end{equation*}   
    which completes the proof of this part in view of \eqref{lem2-eq4}.
    
    \emph{Case 3.2:} $A(\x(t)) > \floor{\tau_{r+1}} + 1 \forall t\in\{0,\ldots,T\}$.
    It holds that $r+1\in\{\r+1,\ldots,\s-1\}$.
    On the other hand, $\x(0)\in\I_{\r,\d}$.
    Thus, in view of \eqref{L},
     \begin{align*}
        &\sum_{i = 1}^{\r} n_i + \sum_{i=1}^{\rr}n'_i +\!\!\! \sum_{i=r+1}^{b}\!\! x_i(0)
        \leq \floor{\tau_{r+1}} + 1.  
    \end{align*}
    As shown in the proof of Lemma \ref{lem15}, $\rr\geq \d$, and hence, $\sum_{i=1}^{\rr}n'_i \geq M$.
    Thus, the above inequality results in
    \begin{align*}
         \sum_{i = 1}^{\r} n_i + M + \sum_{i=r+1}^{b} x_i(0)
        \leq \floor{\tau_{r+1}}+1.
    \end{align*} 
    On the other hand, by the assumption of this case, for $t\leq T$, none of the type-$i$, $i\in\{r+1, \ldots, b\},$ anticoordinators switch to $\A$. 
    Hence, $\sum_{i=r+1}^{b} x_i(T) \leq \sum_{i=r+1}^{b} x_i(0)$.
    Therefore, 
    \begin{align*}
        &\sum_{i = 1}^{\r} n_i + M + \!\!\! \sum_{i=r+1}^{b}\!\! x_i(T)
        \leq \floor{\tau_{r+1}} + 1 \nonumber\\
        \Rightarrow\ 
        & \sum_{i = 1}^{b} x_i(T) 
        \leq \floor{\tau_{r+1}}+1 -M + \!\!\!\sum_{i=\r+1}^r\!\!\! n_i, 
    \end{align*}    
    which completes the proof of this case in view of \eqref{lem2-eq4}.
\end{proof}

\begin{lemma} \label{lem90} 
    Consider $\x(0) \in \I_{\r,\d}$.
    If for some $T\in\mathbb{Z}_{\geq0}$,  
    \begin{equation}   \label{lem90-eq00}
        x^{\n}(t) \geq \r \quad \forall t\in\{0,\ldots, T\},
    \end{equation}
    then 
    \begin{equation} \label{lem90-eq02}
         x^{\o}(t) \leq \ss \quad \forall t\in\{0,\ldots, T\}.
    \end{equation}
\end{lemma}
\begin{proof}
    In view of Lemma \ref{lem2}, the following holds due to \eqref{lem90-eq00}:
    \begin{equation} \label{lem90-eq1}
         \sum_{i=1}^b x_i(t) \leq \sum_{i=1}^b\hat{y}_i
         = \sum_{i=1}^{b'}\tilde{y}_i\quad \forall t\leq T.
    \end{equation}
    On the other hand, $\x(0)\in\I_{\r,\d}$ proves \eqref{lem90-eq02} for $t=0$.
    Now we complete the proof by contradiction.
    Assume that \eqref{lem90-eq02} is violated, for the first time at $t=s$, where $0<s\leq T$.
    Namely,
    \begin{equation} \label{lem90-eq3}
         x^{\o}(t) \leq \ss \quad \forall t\in\{0,\ldots, s-1\},
    \end{equation}
    \begin{equation*} 
         x^{\o}(s) > \ss.
    \end{equation*}
    This implies that some type-$r$, $r\in\{\ss+1,\ldots, b'\}$, coordinator switches to $\A$ at $t=s$, which in view of Lemma \ref{lem_updateRule}, yields 
    \begin{equation} \label{lem90-eq5}
         A(\x(s-1)) \geq \tau'_r \overset{\eqref{temperOrdering}}{\geq} \tau'_{\ss+1}.
    \end{equation}    
    However, in view of \eqref{lem90-eq1} and \eqref{lem90-eq3},
    \begin{equation*} 
         A(\x(s-1)) \leq \sum_{i=1}^{b}\tilde{y}_i +  \sum_{i=1}^{b'}\tilde{y}'_i 
         = A(\tilde{\y}),
    \end{equation*}
    which according to the definition of $\ss$, results in 
    \begin{equation*} 
         A(\x(s-1)) < \tau'_{\ss+1}.
    \end{equation*}
    This contradicts \eqref{lem90-eq5}, completing the proof.
\end{proof}

\begin{lemma} \label{lem9} 
    Consider $\x(0) \in \I_{\r,\d}$.
    If for some $T\in\mathbb{Z}_{\geq0}$,  
    \begin{equation}   \label{lem9-eq01}
        x^{\n}(t) \geq \r \quad \forall t\in\{0,\ldots, T\},
    \end{equation}
    then 
    \begin{equation} \label{lem9-eq02}
         A(\x(t))
         \leq A(\tilde{\y})\quad \forall t\in\{0,\ldots, T\}.
    \end{equation}
\end{lemma}
\begin{proof}
    The result follows Lemmas \ref{lem2} and \ref{lem90} and the fact that $ \sum_{i=1}^b\hat{y}_i \leq  \sum_{i=1}^b\tilde{y}_i$.
\end{proof}

We are ready to complete one of the main steps in proving Theorem \ref{th1}; that is, all type-$1$, \ldots, type-$\r$, are $\A$-fixed.
\begin{lemma}   \label{lem1}
    $
       \x(0)\in\I_{\r,\d} \Rightarrow x_i(t) = n_i\, \forall i\in\{1,\ldots, \r\}\forall t\geq 0.
    $
\end{lemma}
\begin{proof}
    We prove by contradiction. 
    Assume the contrary, and let $T > 0$ denote the first time there exists an $i\in\{1,\ldots, \r\}$ such that $x_i(T) < n_i$.
    Then
    \begin{gather}    
        x^{\n}(t) \geq \r \quad \forall t\leq T-1,  \label{lem1-eq6}\\
        A(\x(T - 1)) > {\tau_i} +1
        \xRightarrow{\eqref{temperOrdering}} A(\x(T - 1)) > {\tau_{\r}} + 1. \label{lem1-eq5}
    \end{gather}
    From \eqref{lem1-eq6} and Lemma \ref{lem9}, 
    $
         A(\x(t))
         \leq A(\tilde{\y})\quad \forall t\leq T-1.
    $
    Hence, since $(\r,\delta)\in\Psi$, it holds that $A(\x(T-1)) \leq {\tau_{\r}}+1$, which contradicts \eqref{lem1-eq5}, completing the proof.
\end{proof}
In the following two lemmas, we show all type-$\ss$, \ldots, type-$b'$ coordinators and type-$\s$, \ldots, type-$b$ anticoordinators are $\B$-fixed.
\begin{lemma}   \label{lem10} 
    $
       x(0)\in\I_{\r,\d} 
        \Rightarrow x'_i(t) = 0\, \forall i\in\{\ss,\ldots, b'\}\forall t\geq 0.
    $
\end{lemma}
\begin{proof}
    Lemma \ref{lem1} implies that \eqref{lem9-eq01} holds for any $T\in\mathbb{Z}_{\geq0}$, implying that so does \eqref{lem90-eq02}, making the result trivial.
\end{proof}
%
\begin{lemma}   \label{lem11} 
    $
       \x(0)\in\I_{\r,\d} 
        \Rightarrow x_i(t) = 0\, \forall i\in\{\s,\ldots, b\}\forall t\geq 0.
    $
\end{lemma}
\begin{proof}
    From Lemma \ref{lem1}, we have that 
    $
        x^{\n}(t) \geq \r\  \forall t\geq 0
    $,
    which in view of Lemma \ref{lem15} yields
    $
        x^{\nn}(t) \geq \d\  \forall t\geq 0
    $.
    Hence, 
    \begin{equation} \label{lem11-eq1}
        A(\x(t)) 
        \geq A(\y) \quad \forall t\geq0.
    \end{equation}
    On the other hand, in view of Lemma \ref{lem_updateRule}, $\s = \overleftarrow{a}^{\o}(\y)$ yields
    $
        A(\y) > \tau_{\s}
    $,
    which according to \eqref{lem11-eq1} results in
    $
        A(\x(t)) 
        > \tau_{\s} \quad \forall t\geq0.
    $    
    Thus, no type $i, i\in\{\s+1, \ldots, b\},$ anticoordinator will ever switch to $\A$, and since none of them were playing $\A$ initially, they always play $\B$.
\end{proof}
We finally show all type-$1$, \ldots, type-$\rr$ coordinators are $\A$-fixed.
\begin{lemma}   \label{lem7}
    $
       \x(0)\in\I_{\r,\d} \Rightarrow x'_i(t) = n'_i \,\forall i\in\{1,\ldots, \rr\}\forall t\geq 0.
    $
\end{lemma}
The proof follows similar steps to those in Lemma \ref{lem1} and is done via the following two lemmas. 
Define $N$ as the total number of $\A$-playing coordinators in the state $\hat{\z}$, i.e., 
$
    N \db \sum_{i = 1}^{\ss-1} n'_i.
$
\begin{lemma}   \label{lem12}
    For $\x(0)\in\I_{\r,\d}$, it holds that
    \begin{equation*}
         \sum_{i=1}^b x_i(t) \geq \sum_{i=1}^b\hat{z}_i \quad \forall t\geq 0. 
    \end{equation*}
\end{lemma}
\begin{proof}
    The proof is similar to that of Lemma \ref{lem2}, which we skip here due to the space limit.
\end{proof}
\begin{proof}[Proof of Lemma \ref{lem7}]
    According to Lemma \ref{lem12},
    \begin{equation} \label{lem91_eq1}
        \sum_{i=1}^b x_i(t) \geq \sum_{i=1}^b \hat{z}_i = \sum_{i=1}^b\tilde{z}_i \quad \forall t\geq 0.
    \end{equation}
    Now we prove by contradiction. 
    Assume the contrary, and let $T > 0$ denote the first time $x'_i(T) < n'_i$ for some $i\in\{1,\ldots, \rr\}$.
    Therefore, 
    \begin{gather}   
        x'_i(t) = n'_i\quad \forall i\in\{1,\ldots, \rr\}\forall t\leq T-1, \label{lem7-eq1} \\
        A(\x(T - 1)) < \tau'_i + 1 \overset{\eqref{temperOrdering}}{\leq} \tau'_{\rr} + 1. \label{lem7-eq5}
    \end{gather}
    However, in view of \eqref{lem91_eq1} and \eqref{lem7-eq1}, 
    $
         A(\x(T-1))
         \geq \sum_{i=1}^b\tilde{z}_i + \sum_{i=1}^{b'}\tilde{z}'_i 
         = A(\tilde{\z}),
    $
    which according to the definition of $\rr$, results in 
    $
         A(\x(T-1)) \geq \tau'_{\rr} + 1.
    $
    This contradicts \eqref{lem7-eq5}, completing the proof.
\end{proof}
Finally, we need to show that the two conditions \eqref{L} and \eqref{R} in the definition of $\I_{\r,\d}$ will not be violated over time.
\begin{lemma}   \label{lem13}
    Let $\x(0) \in \I_{\r,\d}$. 
    Then for all $t\geq 0$,
    \begin{equation*} 
        \sum_{k=1}^{\r} n_k +\! \sum_{k=1}^{\rr} n'_k +\! \sum_{k=i}^{\s-1}\! x_k(t) \leq \floor{\tau_i}+1\  \forall i\in\{\r+1,\ldots,\s-1\}.
    \end{equation*}    
\end{lemma}
\begin{proof}
    We prove by contradiction. 
    Assume the contrary, and let $T\geq 1$ be the first time the inequality is violated by some type-$s,s\in\{\r+1,\ldots,\s-1\},$ anticoordinator.
    Then,
    \begin{gather} 
        \sum_{k=1}^{\r} n_k +\! \sum_{k=1}^{\rr} n'_k +\! \sum_{k=s}^{\s-1}\! x_k(T-1) = \floor{\tau_s}+1, \label{lem13-eq2} \\
        \sum_{k=1}^{\r} n_k +\! \sum_{k=1}^{\rr} n'_k +\! \sum_{k=s}^{\s-1}\! x_k(T) > \floor{\tau_s}+1. \nonumber
    \end{gather}
    Hence, the active agent at time $T-1$ is a $\B$-playing type-$p,p\in\{s,\ldots,\s-1\},$ anticoordinator who switches to $\A$. 
    However, in view of Lemmas \ref{lem1} and \ref{lem7}, 
    \begin{equation*}
        A(\x(T-1)) 
        \geq \sum_{k=1}^{\r} n_k +\! \sum_{k=1}^{\rr} n'_k 
        +\!\!\!\! \sum_{k=\r+1}^{\s-1}\!\!\!\! x_k(T-1)
        \overset{\eqref{lem13-eq2}}{\geq}
        \floor{\tau_s}+1,
    \end{equation*}
    implying that the active agent plays $\B$ at time $T$, which is a contradiction, and the proof is complete.
\end{proof}
\begin{lemma} \label{lem14}
    Let $\x(0) \in \I_{\r,\d}$. 
    Then for all $t\geq 0$,
    \begin{align*}
        &\sum_{k=1}^{\ss-1} n'_k+\sum_{k=1}^{\r} n_k  +\!\!\!\! \sum_{k=\r+1}^{i}\!\!\!\! x_k(t) 
        + \!\!\sum_{k=i+1}^{\s-1}\!\! n_k 
        \geq \floor{\tau_i}+1\\
         &\qquad\qquad\qquad\qquad\qquad\qquad\qquad\   \forall i\in\{\r+1,\ldots,\s-1\}.
    \end{align*}
\end{lemma}
\begin{proof}
    The proof is similar to that of Lemma \ref{lem13}.
\end{proof}
\begin{proof}[Proof of Theorem \ref{th1}]
    The proof follows Lemmas \ref{lem1}, \ref{lem10}, \ref{lem11}, \ref{lem7}, \ref{lem13} and \ref{lem14}.
\end{proof}
\subsection{Bounds on the total number of $\A$-players}
The invariance of $\I_{\r,\d}$ readily confines the total number of $\A$-players in the population. 
\begin{corollary}   \label{cor_exremumNumberOfAPlayers}
    If $\x(0)\in\I_{\r,\d}$, then for all $t\geq 0$,
    \begin{equation*}\scalebox{.89}{\text{$
        \max\left\{\ceil{\tau'_{\rr}}+1, \floor{\tau_{\s}}+1\right\} 
        \leq A(\x(t)) \leq 
        \min\left\{\ceil{\tau'_{\ss}}-1, \floor{\tau_{\r}}+1\right\}. $}}
    \end{equation*}
\end{corollary}
\section{Revisiting Example \ref{example1}} \label{sec_revisiting} 
A positively invariant set is \emph{minimal} if it does not admit a proper subset that is positively invariant. 
According to our simulation results, the population dynamics in Example \ref{example1} possess two minimally positively invariant sets and one equilibrium. 
Here, we use the results in Section \ref{sec_positivelyInvariantSets} to find or approximate these limit sets. 

The benchmarks for the first minimally positively invariant set in the example are 
$\tb{\p = 1}$, $\tr{\q = 5}$, $\tr{\qq = 3}$, $\tb{\pp = 1}$.
Denote this invariant set by $\mathcal{O}_{\tb{1}, \tr{5}, \tr{3}, \tb{1}}$.
For $\r=\p=1$, we obtain $\Delta(1)=\{0,1,3\}$.
Consider the pair $(\r, \d) = (1,0)$, resulting in  
\begin{gather*}
    \overleftarrow{\c}(\overleftarrow{\a}(\y^{1,0})) 
    = (4,2,1,3,15,1,10,2,3)    \\
    \Rightarrow
     A(\overleftarrow{\c}(\overleftarrow{\a}(\y^{1,0}))) = 41 > {\tau_1} + 1 = 19,
\end{gather*}
implying that $(1,0)\not\in\Psi$.

The case with $(\r,\d) = (1,1)$ results in 
\begin{gather*}
    \overleftarrow{\c}(\overleftarrow{\a}(\y^{1,1})) 
    = (4,1,1,1,0,0,0,2,3)\\
    \Rightarrow
    {\tau_{2}} = 9 < A(\overleftarrow{\c}(\overleftarrow{\a}(\y^{1,1}))) = 12 \leq {\tau_1} + 1 = 19, 
\end{gather*}
implying $(1,1)\in\Psi$.
Correspondingly, $\s = \overleftarrow{a}^{\o}(\y^{1,1}) = 5$ as none of the anticoordinators play $\B$, and $\ss = \overleftarrow{c}^{\oo}(\overleftarrow{\a}(\y^{1,1})) = 3$. 
Finally, to obtain $\rr$, we have
\begin{gather*}
    \z^{1,1} = 
    (4,3,1,3,0,0,0,2,3)\\
    \Rightarrow
    \overrightarrow{\a}^{2}(\z^{1,1}) = 
    (4,0,0,0,0,0,0,2,3)\\
    \Rightarrow
    \overrightarrow{\c}(\overrightarrow{\a}^{2}(\z^{1,1})) = 
    (4,0,0,0,0,0,0,0,3).
\end{gather*}
Hence, $\rr= \overrightarrow{c}^{\nn}(\overrightarrow{\a}^{2}(\z^{1,1})) = 1$
Therefore, we obtain the quadruple $(\r,\s,\ss,\rr) = (\tb{1},\tr{5},\tr{3},\tb{1})$, and the set $\I_{\tb{1},\tr{5},\tr{3},\tb{1}}$ is positively invariant according to Theorem \ref{th1}. 
The benchmarks of $\I_{\tb{1},\tr{5},\tr{3},\tb{1}}$ matches those of $\mathcal{O}_{\tb{1}, \tr{5}, \tr{3}, \tb{1}}$, and according to the simulations, $\mathcal{O}_{\tb{1}, \tr{5}, \tr{3}, \tb{1}}$ is the only set with these benchmarks.
On the hand, every positively invariant set includes a subset that is minimally positively invariant.
So $\mathcal{O}_{\tb{1}, \tr{5}, \tr{3}, \tb{1}}\subseteq\I_{\tb{1},\tr{5},\tr{3},\tb{1}} $.
Now we investigate conditions \eqref{L} and \eqref{R} to determine the form of $\I_{\tb{1},\tr{5},\tr{3},\tb{1}}$.
For $i=4,3,2$, condition \eqref{L} becomes
\begin{gather} 
         n'_1 + n_1 + x_4 \leq \floor{\tau_4}+1 
         \Rightarrow 
         x_4 \leq 1, \label{eq_revisitingExample1_1}\\
         n'_1 + n_1 + x_3 + x_4 \leq \floor{\tau_3}+1
         \xRightarrow{\eqref{eq_revisitingExample1_1}} 
         x_3 \leq 1,\nonumber\\
         n'_1 + n_1 + x_2 + x_3 + x_4 \leq \floor{\tau_2}+1
         \Rightarrow
         x_2 + x_3 + x_4 \leq 3,\nonumber
\end{gather}    
where the second inequality is trivial as $x_3\leq n_3 = 1$.
One can verify that the condition \eqref{R} is always satisfied.
Thus, 
\begin{equation*}
    \I_{\tb{1},\tr{5},\tr{3},\tb{1}}
    = \left\{ \x\in\X_{\tb{1}, \tr{5}, \tr{3}, \tb{1}}\,|\,
        x_4\leq1, 
        x_2 + x_3 + x_4 \leq 3
    \right\}.
\end{equation*}
This results in 36 states, and one can check that $\mathcal{O}_{\tb{1}, \tr{5}, \tr{3}, \tb{1}}$ has the same number of states.
Hence,  $\mathcal{O}_{\tb{1}, \tr{5}, \tr{3}, \tb{1}}=\I_{\tb{1},\tr{5},\tr{3},\tb{1}} $.

Now, we proceed to $(\r,\d) = (1,3)$. 
It holds that
\begin{gather*}
    \overleftarrow{\c}(\overleftarrow{\a}(\y^{1,3})) 
    = (4,0,0,0,0,1,10,2,3)\\
    \Rightarrow
    A(\overleftarrow{\c}(\overleftarrow{\a}(\y^{1,3}))) = 20 > {\tau_1} + 1 = 19,
\end{gather*}
implying that $(1,3)\not\in\Psi$.

So the only positively invariant set that our approach finds for when the first benchmark type $\r$ equals one is $\I_{\tb{1},\tr{5},\tr{3},\tb{1}}$.
%
Denote the second minimally positively invariant set in the example by $\mathcal{O}_{\tb{0}, \tr{2}, \tr{5}, \tb{3}}$ with benchmarks
$\tb{\p = 0}$, $\tr{\q = 2}$, $\tr{\qq = 5}$, $\tb{\pp = 3}$.
If we set $\r=\p=0$, we obtain $\Delta(0)=\{3,5\}$.
The pair $(\r, \d) = (0,3)$ results in  
\begin{gather*}
    \overleftarrow{\c}(\overleftarrow{\a}(\y^{0,3})) 
    = (4,0,0,0, 0,1,10,2,3)\\
    \Rightarrow
     {\tau_1} = 18  < 
     A(\overleftarrow{\c}(\overleftarrow{\a}(\y^{0,3}))) = 20 \leq {\tau_0} + 1 = 43,
\end{gather*}
yielding $(0,3)\in\Psi$.
Correspondingly, $\s = \overleftarrow{a}^{\o}(\y^{0,3}) = 2$ and $\ss = \overleftarrow{c}^{\oo}(\overleftarrow{\a}(\y^{0,3})) = 5$. 
To obtain $\rr$, we start from the state
\begin{gather*}
    \z^{0,3} = 
    (4,0,0,0,0,1,10,2,3)\\
    \Rightarrow
    \overrightarrow{\c}(\overrightarrow{\a}^{1}(\z^{0,3})) = 
    (3,0,0,0,0,0,10,2,3).
\end{gather*}
Hence, $\rr= \overrightarrow{c}^{\nn}(\overrightarrow{\a}^{2}(\z^{0,3})) = 3$.
Therefore, we obtain the quadruple $(\r,\s,\ss,\rr) = (\tb{0},\tr{2},\tr{5},\tb{3})$, and $\I_{\tb{0}, \tr{2}, \tr{5}, \tb{3}}$ is positively invariant according to Theorem \ref{th1}, the benchmarks of which again matches those of $\O_{\tb{0}, \tr{2}, \tr{5}, \tb{3}}$.
Thus, $\O_{\tb{0}, \tr{2}, \tr{5}, \tb{3}}\subseteq\I_{\tb{0}, \tr{2}, \tr{5}, \tb{3}}$.
Now we investigate Conditions \ref{L} and \ref{R} to determine the form of $\I_{\tb{0}, \tr{2}, \tr{5}, \tb{3}}$.
Index $i$ in both conditions is limited to $i=1$.
Hence, Condition \ref{L} states that for every $\x\in\I_{\tb{0}, \tr{2}, \tr{5}, \tb{3}}$,
\begin{equation*} 
         \sum_{k=1}^{3} n'_k +\! x_1 \leq \floor{\tau_1}+1 
         \Rightarrow 
         x_1 \leq 4,
\end{equation*}    
which is trivial as $x_1\leq n_1 = 4$,
Also, Condition \ref{R} states
\begin{equation*}
        \sum_{k=1}^{4} n'_k + x_1 
        \geq \floor{\tau_1}+1
        \Rightarrow
        x_1 \geq 3.
\end{equation*}
Thus, $x_1 \in\{3,4\}$.
On the other hand, by definition, $x'_4 \in \{0,1\}$.
This results in the same four states as those in $\mathcal{O}_{\tb{0}, \tr{2}, \tr{5}, \tb{3}}$ (Table \ref{tab:state2}).
That is,
\begin{equation*}
    \I_{\tb{0}, \tr{2}, \tr{5}, \tb{3}}
    = \left\{\x\in\X_{\tb{0}, \tr{2}, \tr{5}, \tb{3}}\,|\,x_1 \in\{3,4\}\right\}
    = \mathcal{O}_{\tb{0}, \tr{2}, \tr{5}, \tb{3}}.
\end{equation*}

The other pair, i.e., $(\r, \d) = (0,5)$, results in the invariant set $\I_{\tb{0}, \tr{1}, \tr{6}, \tb{5}}$.
This set is a singleton that comprises $\x^*$ defined in the example. 
Hence, in this case, our approach results in the equilibrium state of the dynamics.

The cases with $\r\geq2$ do not result in a pair $(\r,\delta)\in\Psi$.
Hence, $\Psi=\{(1,1),(0,3),(0,5)\}$, and the only positively invariant sets that Theorem \ref{th1} puts forward are $\O_{\tb{1}, \tr{5}, \tr{3}, \tb{1}}$, $\O_{\tb{0}, \tr{2}, \tr{5}, \tb{3}}$, and $\{\x^*\}$, all of which coincide with the limit sets of the population dynamics. 
The sets $\I_{\r,\s,\ss,\rr}$, thus, can be minimal. 
The conditions under which they become minimal remain concealed.

\section{Stability analysis}    \label{sec_stability}
Consider the quadruple $(\p,\q,\qq,\pp)\in\Omega$ such that the set $\I_{\p,\q,\qq,\pp}$ defined in Section \ref{sec:example} is positively invariant. 
Note that $(\p,\q,\qq,\pp)$ is not necessarily the same as $(\r,\rr,\ss,\r)$ defined in Section \ref{sec_positivelyInvariantSets}. 
We investigate the notion of stability for the positively invariant set $\I_{\p,\q,\qq,\pp}$.
That is, we examine whether the solution trajectory in the set $\I_{\p,\q,\qq,\pp}$ remains close to the set under small perturbations. 
Unlike in continuous dynamics, here, ``small perturbations'' and ``closeness'' are explicitly quantified to one since the smallest nonzero difference between two states is one, i.e., $\min_{\x,\y\in\X,\x\neq\y}\|\x-\y\| = 1$, where $\|\cdot\|$ is the $L^1$ norm.
The notion of stability, therefore, simplifies to requiring any solution trajectory that is ``adjacent'' to the positively invariant set to remain adjacent or enter the set. 
We provide a formal argument in the following. 

Given a state $\x\in\X$ and set $\O\subseteq\X$, let $\|\x\|_\O$ denote the distant between $\x$ and $\O$, defined by 
$$
    \|\x\|_\O\db\inf_{\y\in\O}\|\x-\y\|.
$$
A natural extension of the notion of equilibrium (Lyapunov) stability yields the following definition of \emph{(positively invariant) set stability}: A positively invariant set $\O\subseteq\X$ is stable if for any $\epsilon>0$, there exists some $\delta>0$ such that if $\|\x(0)\|_\O <\delta$ then $\|\x(t)\|_\O <\epsilon$ for all $t\in\mathbb{Z}_{\geq0}$.
However, since the state space is discrete in our case, the norm $\|\cdot\|$ is confined to non-negative integers.
Thus, we can assume $\epsilon$ and $\delta$ are natural numbers. 
On the other hand, $\epsilon$ cannot be one; otherwise, the definition requires the solution trajectory to always stay in the stable set, even at time zero.
Neither does $\delta$ equal one; otherwise, the solution trajectory starts from the positively invariant set, and hence, will always remain there, which does not allow the investigation of the dynamics under small perturbations.   
Hence, the minimum of $\epsilon$ and $\delta$ is two. 
Moreover, if the stability condition is satisfied for $\epsilon = 2$, then it is also satisfied for $\epsilon > 2$. 
We, therefore, obtain the following definition.
\begin{definition}
    A positively invariant set $\O\subseteq\X$ is stable if under any activation sequence,
    \begin{equation*}
        \|\x(0)\|_\O  \leq 1 
        \Rightarrow
        \|\x(t)\|_\O  \leq 1 \qquad \forall t\in\mathbb{Z}_{\geq0}.
    \end{equation*}
\end{definition}
The following result follows by induction.
\begin{lemma}   \label{lem_stabilityDefinition}
    A positively invariant set $\O\subseteq\X$ is stable if and only if under any activation sequence,
    \begin{equation}        \label{stability_lemma1}
        \|\x(0)\|_\O  = 1 
        \Rightarrow
        \|\x(1)\|_\O  \leq 1.
    \end{equation}
\end{lemma}

So for $\I_{\p,\q,\qq,\pp}$ to be stable, it is necessary and sufficient to investigate solution trajectories that start from adjacent states. 
Let $\bm{1}_i$ be the all-zero vector of length $b+b'$, whose $i^{\text{th}}$ entry is one.
The initial condition $\x(0)$, adjacent to $\I_{\p,\q,\qq,\pp}$, is in the form of $\z\pm\bm{1}_s$ for some state $\z\in\I_{\p,\q,\qq,\pp}$ and type $s$. 
Then the state at the next time step $\x(1)$ will be in the form of $\z\pm\bm{1}_s\pm\bm{1}_v$, where $v$ denotes the type of the initially active agent. 
Stability of $\I_{\p,\q,\qq,\pp}$ requires and implies $\z\pm\bm{1}_s\pm\bm{1}_v$ to be adjacent to $\I_{\p,\q,\qq,\pp}$, which we investigate in Lemmas \ref{lem_stability4} to \ref{lem_stability3}. 
Then we proceed to Proposition \ref{th_stability} as the main result of this section, followed by Theorem \ref{cor_stability}.

Note that a state adjacent to $\I_{\p,\q,\qq,\pp}$, does not belong to it, and hence, violates one of the conditions in the definition of $\I_{\p,\q,\qq,\pp}$, two of which are \eqref{L} and \eqref{R}.
So given the state $\x\in\I_{\p,\q,\qq,\pp}$, we define the sets $\L(\x),\R(\x)\subseteq\{\p+1,\ldots,\q-1\}$ as those indices $i$ that respectively turn inequalities \eqref{L} and \eqref{R} into equality.
These two sets turn out to be key in determining the stability of $\I_{\p,\q,\qq,\pp}$.
We show in Lemma \ref{lem_stability0}, how $\L$ and $\R$ are related.

In Lemmas \ref{lem_stability0} to \ref{lem_stability3}, we consider an arbitrary set $\I_{\p,\q,\qq,\pp}$, where $(\p,\q,\qq,\pp)\in\Omega$, and simplify the notation to $\I$.
We also use $\L = \L(\z)$ and $\R = \R(\z)$ in the proofs.
Moreover, to simplify the results, we assume that the tempers are apart by a distance of at least one, i.e., 
\begin{equation}    \label{assumption_tauOrdering}
    \floor{\tau_i}<\floor{\tau_{i-1}}   \quad \forall i\in\{2,\ldots,b\}.
\end{equation}

\begin{lemma} \label{lem_stability0}
    Given the state $\z\in\I$, if $\R(\z),\L(\z)\neq \emptyset$, then
    $\max\R(\z) \leq \min\L(\z)$.
\end{lemma}
\begin{proof}
    Let $r\in\R$.
    Then
\begin{equation*}
    \sum_{k=1}^{\qq-1} n'_k +\! \sum_{k=1}^{\p} n_k +\!\!\!\! \sum_{k=\p+1}^{r}\!\!\!\! z_k + \!\!\!\sum_{k=r+1}^{\q-1}\!\! n_k
    = \floor{\tau_r}+1.
\end{equation*}
Hence, for $i\in\{\p+1,\ldots, r-1\}$, 
\begin{equation*}
    \sum_{k=1}^{\pp} n'_k +\! \sum_{k=1}^{\p} n_k +\! \sum_{k=i}^{\q-1} z_k
    \leq \floor{\tau_r}+1
    \overset{\eqref{assumption_tauOrdering}}{<} \floor{\tau_{i}}+1.
\end{equation*}
Thus, $\L\cap(\p,r) = \emptyset$.
By setting $r=\max\R$, we obtain $\L\cap(\p,\max\R) = \emptyset$.
Because of $\L\subseteq(\p,\q)$, it follows that $\L\cap(0,\max\R) = \emptyset$.
This completes the proof.
\end{proof}
\begin{lemma}   \label{lem_stability4}
    Consider the state $\z\in\I$, where there exist anticoordinating types $s$ and $v$, satisfying $\p<s<v<\q$, $z_s< n_s$, and $z_v> 0$.
    Define the state $\x = \z + \bm{1}_s - \bm{1}_v$.
    Then $\x\in\I$.
\end{lemma}
\begin{proof}
    The anticoordinating part of $\x$ takes the form
    \begin{equation*}\scalebox{.95}{\text{$
        (\underbrace{n_1,\ldots,n_{\p},z_{\p+1},\ldots,
        z_s+1, \ldots, z_v-1,
        \ldots,z_{\q-1}, 0,\ldots,0}_{\text{anticoordinating}})$}}.
    \end{equation*}  
    Clearly $\x\in\X_{\p,\q,\qq,\pp}$.
    For all $i\in\{\p+1,\ldots,\q-1\}$, 
\begin{align*} 
    &\sum_{k=1}^{\pp} n'_k +\! \sum_{k=1}^{\p} n_k +\! \sum_{k=i}^{\q-1} x_k \nonumber\\
    \leq &\sum_{k=1}^{\pp} n'_k +\!\! \sum_{k=1}^{\p} n_k +\! \sum_{\substack{k=i;\\ k\neq s,v}}^{\q-1}\!\! x_k + (z_s + 1) + (z_v-1)  \nonumber\\
    =& 
    \sum_{k=1}^{\pp} n'_k +\! \sum_{k=1}^{\p} n_k +\! \sum_{k=i}^{\q-1}z_k, \nonumber \\
    \leq&
    \floor{\tau_i}+1.
\end{align*}
where the last inequality follows $\z\in\I$.
So $\x$ satisfies \eqref{L}. 
Similarly, $\x$ satisfies \eqref{R}, leading to the proof. 
\end{proof}
\begin{lemma} \label{lem_stability}
    Consider the state $\z\in\I$, where there exist anticoordinating types $v$ and $s$ satisfying $\p<v<s<\q$, $z_s< n_s$, and $z_v>0$.
    Define the state $\x = \z + \bm{1}_s - \bm{1}_v$.
    Then 
    \begin{equation*}
        |\R(\z)\cap[v,s)| \times |\L(\z)\cap(v,s]| = 0
        \iff
        \|\x\|_{\I} \leq 1,
    \end{equation*}
    with equality when exactly one of $\R(\z)\cap[v,s)$ or $\L(\z)\cap(v,s]$ is empty.
\end{lemma}
\begin{proof}
    The anticoordinating part of $\x$ is of the form
    \begin{equation*}\scalebox{.95}{\text{$
        (\underbrace{n_1,\ldots,n_{\p},z_{\p+1},\ldots,
        z_v-1, \ldots, z_s+1,
        \ldots,z_{\q-1}, 0,\ldots,0}_{\text{anticoordinating}})$}}.
    \end{equation*} 
    The equation $\|\x\|_{\I} =1$ holds if and only if there exits a state $\y\in\I$, such that $\y=\x-\bm{1}_u$ or $\y=\x+\bm{1}_u$ for some $u\in\{1,\ldots,b+b'\}$.
    Now, for $l\in\L$,
    we observe that if $l\in(v,s]$, then $y_u = x_u-1$ and $ u\geq l $.
    %
    Similarly, for $r\in\R$,
    we observe that if $r\in[v,s)$, then $y_u = x_u+1 $ and $ u\leq r $.
    %
    So in the case where both $\L\cap(v,s]$ and $\R\cap[v,s)$ are nonempty, the former requires $y_u=x_u-1$ and the latter requires $y_u=x_u+1$, which is impossible, proving $\|\x\|_{\I} \geq 2$.
    On the other hand, if $\|\x\|_{\I} \geq 2$, then for every state $\y$ satisfying $\|\y-\x\|=1$, it holds that $\y\not\in\I$.
    Hence, by considering $\y = \x - \bm{1}_s$, since $\y\in\X_{\p,\q,\qq,\pp}$ and $\y$ satisfies \eqref{L}, we must have that $\y$ violates \eqref{R}.
    Thus, $\R\cap[v,\q)\neq\emptyset$.
    Similarly, by considering $\y = \x + \bm{1}_v$ we obtain $\L\cap(\p,s]\neq\emptyset$.
    Hence, in view of Lemma \ref{lem_stability0}, both $\L\cap(v,s]$ and $\R\cap[v,s)$ are nonempty.
    Therefore, we proved
    \begin{equation}    \label{lem_stability_1}
        \R\cap[v,s) \neq \emptyset,\L\cap(v,s] \neq \emptyset
        \iff
        \|\x\|_{\I} \geq 2.
    \end{equation}
    
    Now consider the case where $\R\cap[v,s) = \emptyset$ but $\L\cap(v,s] \neq \emptyset$.
    Then in view of Lemma \ref{lem_stability0}, $\R\cap[v,\q)=\emptyset$.
    Now let $\y=\x-\bm{1}_s$.
    Then $\y = \z -\bm{1}_v$, implying that $\y$ satisfies \eqref{L} for $i\in\{\p+1,\ldots,\q-1\}$.
    Moreover, because of $\R\cap[v,\q)=\emptyset$, $\y$ also satisfies \eqref{R} for $i\in\{\p+1,\ldots,\q-1\}$. 
    Therefore, $\y\in\I$.
    On the other hand, $\|\y-\x\| = 1$.
    Hence, $\|\x\|_{\I}  = 1$.
    Namely,
    \begin{equation}    \label{lem_stability_2}
        \R\cap[v,s) = \emptyset, \L\cap(v,s] \neq \emptyset
        \Rightarrow
        \|\x\|_{\I}  = 1.
    \end{equation}
    Similarly, the case with $\R\cap[v,s) = \emptyset, \L\cap(v,s] \neq \emptyset$ can be shown to yield the same result:
    \begin{equation}    \label{lem_stability_3}
        \R\cap[v,s) \neq \emptyset, \L\cap(v,s] = \emptyset
        \Rightarrow
        \|\x\|_{\I}  = 1.
    \end{equation}
    Now if $\|\x\|_{\I}  = 1$, then $\x\not\in\I$, implying that at least one of $\R\cap[v,s)$ and $\L\cap(v,s]$ are nonempty. 
    However, both may not be nonempty as then $\|\x\|_{\I} \neq1$ according to \eqref{lem_stability_1}. 
    Thus, we proved the reverses of \eqref{lem_stability_2} and \eqref{lem_stability_3}, yielding
    \begin{align}   
        &(\R\cap[v,s) = \emptyset, \L\cap(v,s] \neq \emptyset) \nonumber\\
        &\text{ or }
        (\R\cap[v,s) \neq \emptyset, \L\cap(v,s] = \emptyset)
        \iff
        \|\x\|_{\I}  = 1.  \label{lem_stability_4}
    \end{align}
    
    Finally, it is straightforward to show that 
    \begin{equation}    \label{lem_stability_5}
        \R\cap[v,s) = \L\cap(v,s] = \emptyset
        \iff
        x\in\I.
    \end{equation}
    The result follows \eqref{lem_stability_1}, \eqref{lem_stability_4}, and \eqref{lem_stability_5}.
\end{proof}

\begin{lemma} \label{lem_stability2}
    Consider the state $\z\in\I$, where there exist anticoordinating types $s$ and $v$, satisfying $\p<s,v<\q$, $z_s < n_s$, and $z_v < n_v$.
    Define the state $\x = \z + \bm{1}_s + \bm{1}_v$.
    Then 
    \begin{equation*}
        \L(\z)\cap(\p,\min\{v,s\}] = \emptyset
        \iff \|\x\|_{\I} \leq 1.
    \end{equation*}
\end{lemma}
\begin{proof}
    We only prove the case with $v<s$;
    the case with $v\geq s$ can be proven similarly.
    The anticoordinating part of $\x$ is of the form
    \begin{equation*}\scalebox{.95}{\text{$
        (\underbrace{n_1,\ldots,n_{\p},z_{\p+1},\ldots,
        z_v+1, \ldots, z_s+1,
        \ldots,z_{\q-1}, 0,\ldots,0}_{\text{anticoordinating}})$}}.
    \end{equation*}  
    Clearly, $\x$ satisfies \eqref{R}. 
    If $\L\cap[\p+1,v] = \emptyset$, then the state $\y=\x-\bm{1}_s$ satisfies both \eqref{L} and \eqref{R}, and hence, belongs to $\I$ and has a distance of $1$ from $\x$.
    Thus, $\|\x\|_{\I} \leq 1$, proving the necessity part. 
    For sufficiency, $\|\x\|_{\I} \leq 1$ implies the existence of a state $\y\in\I$ such that $\|\y-\x\|\leq 1$. 
    Now if on the contrary, $\L\cap[\p+1,v] \neq \emptyset$, then there exist anticoordinating types $i$ and $j$ such that $\y=\x-\bm{1}_i-\bm{1}_j$. 
    However, this results in a contradiction as then $\|\y-\x\|>1$. 
    Thus, $\L\cap[\p+1,v] = \emptyset$.
\end{proof}

\begin{lemma}   \label{lem_stability3}
    Consider the state $\z\in\I$, where there exists an anticoordinating type $s\in(\p,\q)$ satisfying $z_s < n_s$.
    Let $\x$ be any of the states 
    $\z + \bm{1}_s - \bm{1}_{\alpha}$, where $\alpha\leq\p$ is an anticoordinating type, 
    $\z + \bm{1}_s - \bm{1}_{b+b'-\alpha'+1}$, where $\alpha'\leq\pp$ is a coordinating type, 
    $\z + \bm{1}_s + \bm{1}_{\beta}$, where $\beta\geq\q$ is an anticoordinating type, 
    $\z + \bm{1}_s + \bm{1}_{b+b'-\beta'+1}$, where $\beta'\geq\qq$ is a coordinating type.
    Then 
    $$
    \L(\z)\cap(\p,s] =\emptyset \iff \|\x\|_{\I} \leq 1.
    $$
\end{lemma}
\begin{proof}
    We only prove the case with $\x=\z + \bm{1}_s - \bm{1}_{b+b'-\beta'+1}$ as the remaining cases can be proven similarly.
    The state $\x$ takes the following form:
    \begin{align*}
        \ &(\underbrace{n_1,\ldots,n_{\p}, z_{\p+1},\ldots, z_s+1,\ldots,z_{\q-1}, 0,\ldots, 0}_{\text{anticoordinating}}, \\
        &\qquad\quad\underbrace{0,\ldots,1,\ldots,0, z_{\qq-1},\ldots, z_{\pp+1}, n'_{\pp},\ldots,n'_1}_{\text{coordinating}}).
    \end{align*}
    Clearly $\|\x\|_{\I}\neq 0$ as $\x\not\in\X_{\p,\q,\qq,\pp}$. 
    If $\|\x\|_{\I} = 1$, there exists $\y\in\I$ such that $\|\x-\y\|= 1$. 
    Then $\y$ must be $\x-\bm{1}_{b+b'-\beta'+1}$ to satisfy $\y\in\X_{\p,\q,\qq,\pp}$.
    Hence, $\y=\z+\bm{1}_s$.
    Then, since $\y$ satisfies \eqref{L}, $\L\cap(\p,s]=\emptyset$.
    This proves the necessity. 
    The same $\y$ proves sufficiency. 
\end{proof}

Given the state $\z\in\I_{\p,\q,\qq,\pp}$, let $z^{\e}$ denote the maximum wandering anticoordinating type whose agents are not all playing $\A$, i.e., 
$
    z^{\e} \db \max\{i\in(\p,\q)\,|\, z_i<n_i\},
$
where we define $\max\emptyset = \p$.
Similarly, define
$
    z^{\f} \db \min\{i\in(\p,\q)\,|\, z_i>0\},
$
where we define $\min\emptyset = \q$.
Given $\z\in\X$, define $z_\W$ as the number of $\A$-playing wandering anticoordinators at $\z$, i.e., 
$
    z_\W \db \sum_{i=\p+1}^{\q-1}z_i,
$
and let $n_\W$ denote the total number of wandering anticoordinators, i.e., 
$
    n_\W \db \sum_{i=\p+1}^{\q-1}n_i.
$
Consider the state $\z\in\I_{\p,\q,\qq,\pp}$, where $z_\W<n_\W$, and let $s\in(\p,\q)$ be a wandering anticoordinating type that satisfies $z_s<n_s$.
Define $w^{\B}(\z)$ as the smallest wandering anticoordinating type, the agents of which switch to $\B$ upon activation at $\z+\bm{1}_s$:
\begin{equation*}
    w^{\B}(\z) \db
    \min\left\{i\in(\p,\q) \,|\, z_i>0,
    A(\z) > {\tau_{i}}
    \right\}.
\end{equation*}
Similarly, for $\z\in\I_{\p,\q,\qq,\pp}$, where $z_\W<n_\W$, define 
\begin{equation*}
    w^{\A}(\z) \db
    \max\left\{i\in(\p,\q) \,|\, z_i<n_i,
    A(\z) \leq {\tau_{i}}-1
    \right\}.
\end{equation*}
For the case where $z_\W=n_\W$,
we define $w^{\A}(\z) = \p$ and $w^{\B}(\z) = \q$.
It follows that $w^{\A}(\z) < w^{\B}(\z)$. 
Consider the state $\z\in\I_{\p,\q,\qq,\pp}$, where $z_\W>0$, and let $s\in(\p,\q)$ be a wandering anticoordinating type that satisfies $z_s>0$.
Define $v^{\B}(\z)$ as the smallest wandering anticoordinating type, the agents of which switch to $\B$ upon activation at the state $\z-\bm{1}_v$:
\begin{equation*}
    v^{\B}(\z) \db
    \min\left\{i\in(\p,\q) \,|\, z_i>0,
    A(\z) > {\tau_{i}}+2
    \right\}.
\end{equation*}
Similarly, for $\z\in\I_{\p,\q,\qq,\pp}$, where $z_\W>0$, define 
\begin{equation*}
    v^{\A}(\z) \db
    \max\left\{i\in(\p,\q) \,|\, z_i<n_i,
    A(\z) \leq {\tau_{i}}+1
    \right\}.
\end{equation*}
For the case where $z_\W=0$, define $v^{\A}(\z)=\p$ and $v^{\B}(\z)=\q$.
Recall that $\tau_0=n$ and $\tau'_0=-2$.
We further define $\tau_{b+1}=-2$ and $\tau'_{b'+1}=n+2$.
\begin{proposition} \label{th_stability}
    A positively invariant set $\I_{\p,\q,\qq ,\pp }\subseteq\X$, where $(\p,\q,\qq,\pp)\in\Omega$, is stable if and only if for every state $\z\in\I_{\p,\q,\qq ,\pp }$, the following statements hold:
    \begin{enumerate}
        \item if 
            $A(\z)=\floor{\tau_{\p}}+1 \text{ or }\ceil{\tau'_{\qq}}-1$, then
            $\L(\z)\cap(\p,z^{\e}] = \emptyset,$
        \item if
            $A(\z)=\floor{\tau_{\q}}+1 \text{ or } \ceil{\tau'_{\pp}}+1$, then
            $\R(\z)\cap[z^{\f},\q) = \emptyset,$
        \item if $\q+\qq\leq b+b'+1$, then 
        \begin{equation} \label{stabilityTheorem_e1}
            A(\z) \leq     \min\{\floor{\tau_{\p}},\ceil{\tau'_{\qq}}-2\},
        \end{equation}
        \begin{equation} \label{stabilityTheorem_e2}
            \R(\z)\cap[w^{\B}(\z),\q) = \emptyset,
        \end{equation}
        \item if $\p+\pp\geq 1$, then 
        \begin{equation} \label{stabilityTheorem_e3}
            A(\z) \geq \max\{\floor{\tau_{\q}}+2,\ceil{\tau'_{\pp }}+2\},
        \end{equation}
        \begin{equation}\label{stabilityTheorem_e4}
            \L(\z)\cap(\p,v^{\A}(\z)] = \emptyset,
        \end{equation}
        \item if $z_\W<n_\W$, then
        $$
        |\R(\z)\cap[w^{\B}(\z),z^{\e})| \times |\L(\z)\cap(w^{\B}(\z),z^{\e}]| = 0,
        $$
        and if $z_\W>0$, then
        $$ 
         |\R(\z)\cap[z^{\f},v^{\A}(\z))| \times |\L(\z)\cap(z^{\f},v^{\A}(\z)]| = 0,
        $$
        \item 
        if $z_\W<n_\W-1$ or $\q+\qq\leq b+b'+1$, then
        \begin{equation}\label{stabilityTheorem_e5}
        \L(\z)\cap(\p,w^{\A}(\z)] = \emptyset,
        \end{equation}
        \item 
        if $z_\W>1$ or $\p+\pp\geq 1$, then
        \begin{equation} \label{stabilityTheorem_e6}
            \R(\z)\cap[v^{\B}(\z),\q) = \emptyset.
        \end{equation}      
    \end{enumerate}
\end{proposition}
\begin{proof}
(necessity) 
Let $\I$ be stable. 
Consider the initial state $\x(0)\in\X$ that satisfies $\|\x(0)\|_{\I} = 1 $. 
Then there exists some state $\z\in\I$ such that $\|\x(0) - \z\| = 1$.
Because of $\|\x(0)\|_{\I} = 1 $ and by the definition of $\I$, one of the following cases is in force.
According to Lemma \ref{lem_stabilityDefinition}, stability of $\I$ yields $\|\x(1)\|_{\I}\leq 1$ in all of the cases.

\emph{Case 1:} $\x(0) \in \X_{\p,\q,\qq ,\pp }$ and $\x(0)$ violates \eqref{L}.
Then $\x(0)=\z+\bm{1}_s$ for some $s\in(\p,\q)$, and $\x(0)$ takes the form
\begin{equation*}\scalebox{1}{\text{$
    (\underbrace{n_1,\ldots,n_{\p},z_{\p+1},\ldots,z_s+1,\ldots,z_{\q-1}, 0,\ldots,0}_{\text{anticoordinating}}, 
    \underbrace{*, \ldots, *}_{\text{coordinating}})$}}.
\end{equation*}
Denote the type of the initially active agent by $v$.
If $v$ is coordinating and $v \leq \pp $, then
\begin{equation*}
    A(\x(0))
    =A(\z)+1
    \geq {\tau'_{\pp }}+2,
\end{equation*}
where the last inequality is due to the invariance of $\I$.
Hence, the active agent does not switch strategies.
The same holds when $v$ is anticoordinating and $v\geq \q$.
On the other hand, if the initially active agent does not switch strategies, $\|\x(1)\|_{\I}=1$, making the result trivial.
Otherwise, one of the following sub-cases holds.

\emph{Case 1.1:} $v$ is anticoordinating and $v\leq\p$, or $v$ is coordinating and $v\geq\qq$.
Then respectfully $A(\x(0))-1>\tau_{\p}$ or $A(\x(0))\geq\tau'_{\qq}$ must hold for the active agent to switch strategies. 
Equivalently, $A(\z)>\tau_{\p}$ or $A(\z)\geq\tau'_{\qq}-1$.
On the other hand, invariance of $\I$ implies $A(\z)\leq\tau_{\p}+1$ and $A(\z)<\tau'_{\qq}$.
Thus, the two conditions become $A(\z)=\floor{\tau_{\p}}+1$ or $A(\z)=\ceil{\tau'_{\qq}}-1$.
Hence, in view of Lemma \ref{lem_stability3} and $\|\x(1)\|_{\I}\leq1$,
\begin{equation} \label{stabilityTheorem_eq22}
    A(\z)=\floor{\tau_{\p}}+1 \text{ or }\ceil{\tau'_{\qq}}-1
    \Rightarrow
    \L(\z)\cap(\p,s] = \emptyset.
\end{equation}

\emph{Case 1.2:}   $v$ is anticoordinating,  $v\in[w^{\B}(\z),\q)$, and $x_v(0)>0$. 
Then the active agent will switch to $\B$, resulting in $\x(1) = \z - \bm{1}_v + \bm{1}_s$.
If $v \geq s$, then in view of Lemma \ref{lem_stability4}, $\x(1)\in\I$.
If $v<s$, then in view of Lemma \ref{lem_stability},
\begin{equation} \label{stabilityTheorem_eq3}
     |\R(\z)\cap[v,s)| \times |\L(\z)\cap(v,s]| = 0.
\end{equation}

\emph{Case 1.3:}   $v$ is anticoordinating,  $v\in(w^{\A}(\z),w^{\B}(\z))$.
Then if $v=s$, we obtain $\x(1)=\x(0)$.
Otherwise, the active agent will not switch.

\emph{Case 1.4:}  $v$ is anticoordinating,
$v\in(\p,w^{\A}(\z)]$, and $x_v(0) <n_v$. 
Then the active agent will choose $\A$, resulting in $\x(1) = \z + \bm{1}_v + \bm{1}_s$.
Hence, according to Lemma \ref{lem_stability2},
\begin{equation}    \label{stabilityTheorem_eq4}
    \L(\z)\cap(\p,\min\{v,s\}] = \emptyset.
\end{equation}

\emph{Case 1.5:} $v$ is coordinating and $v \in(\pp,\qq)$.
Then regardless of the strategy the active agent chooses, the state $\y=\x(1)-\bm{1}_s$ belongs to $\I$, resulting in $\|\x(1)\|_{\I} =1$.

In view of Lemma \ref{lem_stabilityDefinition}, stability of $\I$ implies \eqref{stabilityTheorem_eq22}, \eqref{stabilityTheorem_eq3}, and \eqref{stabilityTheorem_eq4} are satisfied for any initially active agent and any initial state $\x(0)$ that falls into Case 1.
The existence of $s$ in Case 1 implies $z_\W<n_\W$.
Moreover, the maximum value that $s$ can take is $z^{\e}$.
Hence, \eqref{stabilityTheorem_eq22} yields that if $z_\W<n_\W$, then
\begin{equation} \label{stabilityTheorem_eq30}
    A(\z)=\floor{\tau_{\p}}+1 \text{ or }\ceil{\tau'_{\qq}}-1
    \Rightarrow
    \L(\z)\cap(\p,z^{\e}] = \emptyset.
\end{equation} 
Now if $z_\W=n_\W$, then $(\p,z^{\e}] = (\p,\p] = \emptyset$, making the condition trivial. 
Hence, \eqref{stabilityTheorem_eq30} can be stated without the condition $z_\W<n_\W$.
In \eqref{stabilityTheorem_eq3}, the minimum value that $v$ can take is $w^{\B}(\z)$.
Thus, we obtain that if $z_\W<n_\W$, then 
\begin{equation}   \label{stabilityTheorem_eq31}
         |\R(\z)\cap[w^{\B}(\z),z^{\e})| \times |\L(\z)\cap(w^{\B}(\z),z^{\e}]| = 0.
\end{equation}
For \eqref{stabilityTheorem_eq4}, both type-$s$ and $v$ agents switch to $\A$, implying $z_\W<n_\W-1$.
On the other hand, $\min\{w^{\A}(\z),z^{\e}\}=w^{\A}(\z)$.
Thus, we obtain that if $z_\W<n_\W-1$, then
\begin{equation}    \label{stabilityTheorem_eq32}
    \L(\z)\cap(\p,w^{\A}(\z)] = \emptyset.
\end{equation}

\emph{Case 2:} 
$\x(0) \in \X_{\p,\q,\qq ,\pp }$ and $\x(0)$ violates \eqref{R}.
Then $\x(0)=\z-\bm{1}_s$ for some $v\in(\p,\q)$, and $\x(0)$ takes the form
\begin{equation*}\scalebox{1}{\text{$
    (\underbrace{n_1,\ldots,n_{\p},z_{\p+1},\ldots,z_s-1,\ldots,z_{\q-1}, 0,\ldots,0}_{\text{anticoordinating}}, 
    \underbrace{*, \ldots, *}_{\text{coordinating}})$}}.
\end{equation*}
Then similar to the previous case, the followings can be concluded:
if $z_\W>0$, then
\begin{gather} 
A(\z)=\floor{\tau_{\q}}+1 \text{ or } \ceil{\tau'_{\pp}}+1
    \Rightarrow
    \R(\z)\cap[z^{\f},\q) = \emptyset, \label{stabilityTheorem_eq33}\\
         |\R(\z)\cap[z^{\f},v^{\A}(\z))| \times |\L(\z)\cap(z^{\f},v^{\A}(\z)]| = 0. \label{stabilityTheorem_eq34}
\end{gather}
Moreover, if $z_\W>1$, then
\begin{equation}    \label{stabilityTheorem_eq35}
    \R(\z)\cap[v^{\B}(\z),\q) = \emptyset.
\end{equation}

\emph{Case 3:} $\x(0) \not\in \X_{\p,\q,\qq ,\pp }$. 
Denote the type of the initially active agent by $v$.
Then either of the followings holds.

\emph{Case 3.a:} 
$\x = \z + \bm{1}_s$ for some anticoordinating type $s\geq \q$ or coordinating type $s\geq\qq$.  
If $v$ is anticoordinating and $v\leq\p$ or $v$ is coordinating and $v\leq\pp$, then $\|\x(1)\|_{\I} \leq1$ implies that the active agent does not switch to $\B$.
Thus, $\ceil{\tau'_{\pp }} \leq A(\x(0))-1 \leq \floor{\tau_{\p}}$, or equivalently, $\ceil{\tau'_{\pp }} \leq A(\z) \leq \floor{\tau_{\p}}$, where the left condition is already implied by the invariance of $\I$.
Hence, 
$
    A(\z) \leq \floor{\tau_{\p}}.
$    
If $v$ is anticoordinating and $v\geq\q$ or $v$ is coordinating and $v\geq\qq$, then $\|\x(1)\|_{\I} \leq1$ implies that the active agent does not switch to $\A$; equivalently, $\floor{\tau_{\q}}+1 \leq A(\x(0))\leq \ceil{\tau'_{\qq }}-1$, where the left condition is already implied by the invariance of $\I$.
Hence, 
$    
    A(\z)\leq \ceil{\tau'_{\qq }}-2.    
$    
If $v$ is coordinating and $v\in(\pp,\qq)$, then regardless of the active agent's choice at time $1$, the state $\y=\x(1)-\bm{1}_s$ belongs to $\I$ and satisfies $\|\y-\x(1)\|\leq1$, implying $\|\x(1)\|_{\I} \leq 1$.
Finally, if $v$ is anticoordinating and $v\in(\p,\q)$, then either 
$v\in[w^{\B}(\z),\q)$ and $x_v(0)>0$, resulting in 
$
    \R(\z)\cap[v,\q) = \emptyset,
$    
or $v\in(\p,w^{\A}(\z)]$ and $x_v(0) <n_v$, resulting in 
$
    \L(\z)\cap(\p,v] = \emptyset.
$    
So this case can be summarized as
if $\q\leq b$ or $\qq\leq b'$, then
\begin{gather}    
    A(\z) \leq \min\{\floor{\tau_{\p}},\ceil{\tau'_{\qq }}-2\}, \label{stabilityTheorem_eq36} \\
    \L(\z)\cap(\p,w^{\A}(\z)] = \emptyset, \label{stabilityTheorem_eq37} \\
    \R(\z)\cap[w^{\B}(\z),\q) = \emptyset. \label{stabilityTheorem_eq37} 
\end{gather}

\emph{Case 3.b:} 
$\x = \z - \bm{1}_s$ for some anticoordinating type $s\leq \p$ or coordinating type $s\leq\pp$.  
Similar to the previous case, here we conclude that
if $\p\geq 1$ or $\pp\geq 1$, then
\begin{gather}    
    A(\z) \geq \max\{\floor{\tau_{\q}}+2,\ceil{\tau'_{\pp }}+2\}. \label{stabilityTheorem_eq39} \\
    \L(\z)\cap(\p,v^{\A}(\z)] = \emptyset.\label{stabilityTheorem_eq40} \\
    \R(\z)\cap[v^{\B}(\z),\q) = \emptyset.\label{stabilityTheorem_eq41} 
\end{gather}
The proof then follows \eqref{stabilityTheorem_eq30} to
\eqref{stabilityTheorem_eq41}.

(sufficiency) Starting from an arbitrary initial state $\x(0)$ satisfying $\|\x(0)\|_{\I} = 1$, and under any initial condition, $\x(1)$ ends up at one of the above three cases, the conditions of which are all fulfilled as the seven conditions in the Theorem are in force. 
Therefore, $\|\x(1)\|_{\I}\leq 1$, and hence, $\I$ is stable.
\end{proof}


The conditions $\p+\pp\geq 1$ and $\q+\qq \leq b+b'+1$ in Proposition \ref{th_stability} are weak as they imply that the invariant set has at least one $\A$-fixed and one $\B$-fixed type.
When the conditions are satisfied, the proposition is simplified to the following.
\begin{theorem} \label{cor_stability}
    A positively invariant set $\I_{\p,\q,\qq ,\pp }\subseteq\X$, where $(\p,\q,\qq,\pp)\in\Omega$, $\p+\pp\geq 1$, and $\q+\qq \leq b+b'+1$, is stable if and only if for every state $\z\in\I_{\p,\q,\qq ,\pp }$, the followings hold:
    \begin{equation}    \label{cor_stability_e1}
        A(\z)\in[\max\{\floor{\tau_{\q}},\ceil{\tau'_{\pp }}\}+2, \min\{\floor{\tau_{\p}},\ceil{\tau'_{\qq }}-2\}],
    \end{equation}
    \begin{equation}    \label{cor_stability_e2}
        \L(\z)\cap(\p,v^{\A}(\z)] 
        =\R(\z)\cap[w^{\B}(\z),\q)
        =\emptyset.
    \end{equation}
\end{theorem}
\begin{proof}
    Equation \eqref{cor_stability_e1} yields
    \eqref{stabilityTheorem_e1} and \eqref{stabilityTheorem_e3} in
    Proposition \ref{th_stability}, which in turn imply that the first two statements in the proposition are satisfied. 
    On the other hand, $w^{\A}(\z)\leq v^{\A}(\z)$ and $w^{\B}(\z)\leq v^{\B}(\z)$.
    Hence, \eqref{cor_stability_e2} yields  \eqref{stabilityTheorem_e2}, \eqref{stabilityTheorem_e4}, \eqref{stabilityTheorem_e5}, and \eqref{stabilityTheorem_e6}.
    Thus, the third, fourth, sixth, and seventh statements also hold. 
    On the other hand, \eqref{stabilityTheorem_e2} and \eqref{stabilityTheorem_e4} cover the fifth statement in the proposition, completing the proof.
\end{proof}

\section{Revisiting the example: II}    \label{sec_revisitingExample2}
Consider the first positively invariant set $\I_{\tb{1},\tr{5},\tr{3},\tb{1}}$.
Given the initial state $\x(0)=(3,0,0,1,0,0,0,0,3)$ and when the initially active agent is a type-4 $\A$-playing anticoordinator, we obtain $\x(1)=(3,0,0,2,0,0,0,0,3)$.
It can be shown that $\|\x(1)\|_{\I_{\tb{1},\tr{5},\tr{3},\tb{1}}}=2$; hence, $\I_{\tb{1},\tr{5},\tr{3},\tb{1}}$ is unstable. 
This is can be also verified by Theorem \ref{cor_stability} as for the state $\z=\x(0)$ above, $\L(\z)=\{4\}$ which intersects $(p,v^{\A}(\z)]=(1,4]$.

For the second positively invariant set $\I_{\tb{0}, \tr{2}, \tr{5}, \tb{3}}$, Condition \eqref{cor_stability_e1} becomes
$
    A(\z)\in[\max\{9,14\}+2, \min\{42,23\}]=[16,23],
$
which is satisfied as $A(\z)\in[18,20]$ for all $\z\in\I_{\tb{0}, \tr{2}, \tr{5}, \tb{3}}$.
On the other hand, the only states that result in a nonempty $\L$ are
$\z^1=({\color{Mygreen}4}, {\color{Myred}0}, {\color{Myred}0}, {\color{Myred}0}, {\color{Myred}0}, {\color{Mygreen}0}, {\color{Myblue}10}, {\color{Myblue}2}, {\color{Myblue}3})$ and  
$\z^2=({\color{Mygreen}4}, {\color{Myred}0}, {\color{Myred}0}, {\color{Myred}0}, {\color{Myred}0}, {\color{Mygreen}1}, {\color{Myblue}10}, {\color{Myblue}2}, {\color{Myblue}3})$, yielding $v^{\A}(\z^1) = v^{\A}(\z^2) = \max\emptyset=0$.
Similarly, the only state that results in a nonempty $\R$ is 
$\z^3=({\color{Mygreen}3}, {\color{Myred}0}, {\color{Myred}0}, {\color{Myred}0}, {\color{Myred}0}, {\color{Mygreen}1}, {\color{Myblue}10}, {\color{Myblue}2}, {\color{Myblue}3})$, yielding
$w^{\B}(\z^3) = \min\emptyset = 2$.
Hence, 
$
    \L(\z)\cap(0,0] =
    \R(\z)\cap[2,2) = \emptyset,
$ for all $\z\in\I_{\tb{0}, \tr{2}, \tr{5}, \tb{3}}$,
implying the stability of $\I_{\tb{0}, \tr{2}, \tr{5}, \tb{3}}$.

Finally, the equilibrium $\x^*$ satisfies the first condition of Theorem \ref{cor_stability} as $\A(\x^*)=31\in[27,42]$. 
The second condition is also satisfied since $\L(\z)=\R(\z)=\emptyset$.
Hence, $\x^*$ is stable.
\section{Synchronous updates}   \label{sec_synchronousUpdates}
The agents in our model update their strategies \emph{asynchronously}; that is, a single agent becomes active at each time step. 
For example, individuals may separately decide to join or leave NGOs at different times, based on their availability, and companies may change their strategies separately, based on their financial situation. 
Although often less realistic, one may consider the \emph{synchronous} case, where all agents update their strategies simultaneously at every time step \cite{lopez2008social}. 
This leaves no room for the possible stochasticity in the agents' activation sequence, resulting in fully deterministic dynamics.

More specifically, given the population state $\x(t)$, the state at the next time step is obtained by $\x(t+1)=\b(\x(t))$, where $\b:\X\to\X$ is the \emph{best-response map} defined by 
\begin{equation*}\scalebox{.86}{\text{$
    \beta_i(\x) = 
    \begin{cases}
        n_i     & A(\x)\leq \tau_i \\
        x_i     & A(\x) = \floor{\tau_i} +1\\
        0       & A(\x) > \tau_i+1
    \end{cases},
    \beta_j(\x) = 
    \begin{cases}
        n'_j    & A(\x)\geq \tau'_j+1 \\
        n'_j-x'_j    & A(\x) = \ceil{\tau'_j}\\
        0       & A(\x) < \tau'_j
    \end{cases}$}}
\end{equation*}
where $i\in\{1,\ldots,b\}$ is an anticoordinating and $j\in\{1,\ldots,b'\}$ is a coordinating type. 
The equilibria of the synchronous and asynchronous dynamics are the same since both require every individual to be satisfied with their strategies. 
What about the invariant sets?
Under the synchronous dynamics, same-type agents either all simultaneously or none switch strategies. 
Hence, the population state no longer evolves gradually by a distance of one from the previous state, but may undergo large jumps at each time step. 
As an immediate result, \eqref{L} and \eqref{R} may no longer hold, and the rich asynchronous dynamics may become limited to a small limit cycle. 

To analyze the synchronous dynamics, Granovetter \cite{granovetter1978threshold} (followed by \cite{granovetter1983threshold}) proposed the convenient approach of focusing on the evolution of the number of $\A$-players, $A(t)$, rather than the state $\x(t)$, and obtain a first-order difference equation $A(t+1)=F(A(t))$ for some function $F$. 
However, this is not the case with our model, because $A(t+1)$ is not a function of $A(t)$ only but also $\x(t)$ as seen in the definition of $\b$.
Nevertheless, due to the potential interesting insights of this approach, we consider the simplified population dynamics where $A_j$ in \eqref{eq:coor} and \eqref{eq:anti} is replaced with $A$; namely, instead of basing her decisions on the number of $\A$-players in the remaining of the population, agent $j$ bases her decision on the total number of $\A$-players in the population.
This simplifies $\b$ to $\hat{\b}$ as follows:
\begin{equation*}\scalebox{.91}{\text{$
    \hat{b}_i(\x) = 
    \begin{cases}
        n_i     & A(\x)\leq \tau_i \\
        0       & A(\x) > \tau_i
    \end{cases},\
    \hat{b}_j(\x) = 
    \begin{cases}
        n'_j    & A(\x)\geq \tau'_j \\
        0       & A(\x) < \tau'_j
    \end{cases}$}.}
\end{equation*}
Then the number of $\A$-playing coordinators, denoted by $A^c$, is governed by $A^c(t+1) = F^c(A(t))$, where $F^c$ is the cumulative distribution function of the coordinators' tempers. 
Moreover, the number of $\A$-playing anticoordinators, denoted by $A^a$, is governed by $A^a(t+1) = F^a(A(t))$, where $F^a(A)$ is the number of anticoordinators whose tempers are non-less than $A$.
Thus, the total number of $\A$-players is governed by 
\begin{equation}    \label{eq_synchronousDynamics}
        A(t+1) = F(A(t)),
\end{equation}
where $F=F^c+F^a$, which is plotted in Figure \ref{Fig2} for the population in Example \ref{example1}. 
\begin{figure}[h]
	\centering	\includegraphics[width=1\linewidth]{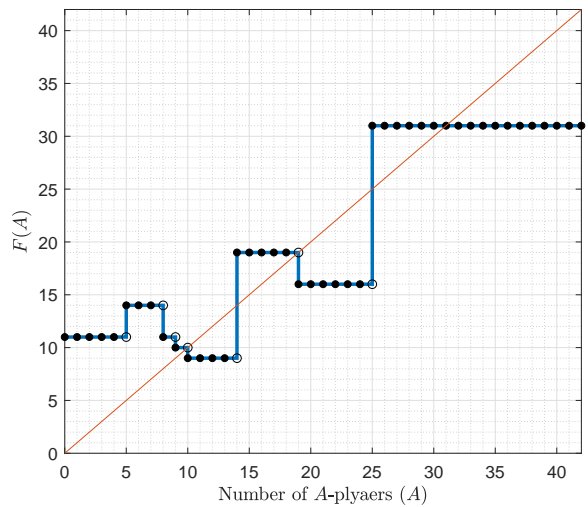}
    \caption{The evolution function of the modified synchronous population dynamics \eqref{eq_synchronousDynamics}.}
	\label{Fig2}
\end{figure} 
Intersections of $F$ with the line $F(A)=A$ results in the equilibria. 
As expected, the point $(A,F(A))=(31,31)$, is an equilibrium of the synchronous dynamics and correctly matches the number of $\A$-players at the equilibrium $\x^*$ that the asynchronous dynamics admit. 
The synchronous dynamics additionally admit two limit cycles, characterized by $\C_1=\{9,10\}$ and  $\C_2=\{16,19\}$, roughly corresponding to the first and second invariant sets $\O_{\tb{1}, \tr{5}, \tr{3}, \tb{1}}$ and $\O_{\tb{0}, \tr{2}, \tr{5}, \tb{3}}$. 
One can also easily find the domain of attractions from the graph: 
the domain of attraction of $\C_1$ is $\Phi_1 = \{1,\ldots,4,8,\ldots,13\}$, that of $\C_2$ is $\Phi_2 = \{5,6,7,14\ldots,24\}$, and that of the equilibrium is $\Phi^*=\{25,\ldots,42\}$.
It follows that all three are stable under the dynamics \eqref{eq_synchronousDynamics}.
Moreover, the equilibrium has the largest domain of attraction, and hence, is the most probable long-term outcome of the dynamics, given an arbitrary initial condition.

So the dynamics of the number of $\A$-players under synchronous updating provides useful insights for and may be considered as an ``approximation'' of the asynchronous dynamics; however, there are fundamental differences. 
For example, unlike the asynchronous case, here the limit cycles are only of length two, and the extremum number of $\A$-players in the limit cycles and invariant sets do not match. 
Also, the first limit cycle is stable under \eqref{eq_synchronousDynamics}, whereas the first positively invariant set is unstable under the asynchronous dynamics. 
Moreover, in the first limit cycle $\C_1$, type 3 and 4 anticoordinators are $\B$-fixed and never switch strategies, whereas they wander in the asynchronous dynamics. 
One may also note the somewhat surprising gap in the domain of attraction $\Phi_1$, where for $A=1,\ldots,4$ and $A=8,\ldots,13$, the solution trajectory reaches the first limit cycle, but for $A=5,6,7$, it reaches the second.
The situation is not much different under the original synchronous dynamics governed by $\b$. 
Our simulations show that then $\C_1$ and $\C_2$ will each include an additional state: $\{9,10,12\}$ and $\{16,19,20\}$, which are still different from the invariant sets in Example 1. 
Overall, the exact relationship between the synchronous and asynchronous dynamics remains unknown.

\section{Conclusion}
We have studied a well-mixed heterogeneous population of coordinators and anticoordinators.
We have explicitly characterized a class of positively invariant sets under the population dynamics.
If any such set does not include an equilibrium, it contains a minimally positively invariant subset where the solution trajectory will perpetually fluctuate and never equilibrate. 
In the context of collective decision-makers such as NGOs and firms, this means that observed oscillations in the decisions are not necessarily caused by errors in the decisions; namely, they are not a matter of chance. 
Nor do they imply a complex interaction network between the individuals. 
They are likely intrinsic to the dynamics.
Moreover, the oscillations are not unpredictable. 
For example, Corollary \ref{cor_exremumNumberOfAPlayers} bounds the number of individuals choosing a certain action over time. 
We also know if the oscillations are stable, and hence, persistent under small errors in the individuals' decisions.
These results pave the path to control the corresponding populations, for example, by providing incentives \cite{riehl2018incentive}, and lead them to a desired outcome where the oscillations are settled or modified.
Designing the control algorithms as well as fitting the dynamics to data remain as future work.



\section*{Appendix}
The following results are straightforward.
We simplify the notation $A(\x(t))$ to $A(t)$.
\begin{lemma} \label{lem_updateRule}
    A type-$i$ anticoordinator playing $\A$ tends to play $\A$ (equivalently, does not tend to play $\B$) at time $t+1$ iff $A(t) \leq \tau_i +1$, 
    and tends to play $\B$ iff $A(t) > \tau_i +1$.
    A type-$i$ anticoordinator playing $\B$ tends to play $\B$ iff $A(t)> \tau_i$, and tends to play $\A$ iff $A(t)\leq  \tau_i$.
    A type-$i$ coordinator playing $\A$ tends to play $\A$ iff $A(t) \geq \tau'_i +1$, and tends to play $\B$ iff $A(t) < \tau'_i +1$. 
    A type-$i$ coordinator playing $\B$ tends to play $\A$ iff $A(t) \geq \tau'_i$, and tends to play $\B$ iff $A(t) < \tau'_i$. 
\end{lemma}
\begin{lemma} \label{lem_updateRule2}
    A type-$i$ anticoordinator tends to play $\A$ at time $t+1$ if $A(t) \leq \tau_i$, 
    and tends to play $\B$ if $A(t) > \tau_i +1$.
    A type-$i$ coordinator tends to play $\A$ at time $t+1$ if $A(t) \geq \tau'_i +1$, and tends to play $\B$ if $A(t) < \tau'_i$.
\end{lemma}

	\bibliographystyle{IEEEtran}
	\bibliography{bib}%
\begin{IEEEbiography}
{Pouria Ramazi}
is currently an assistant professor at the Department of Mathematics and Statistics at Brock University, Canada. 
	He received the B.S. degree in electrical engineering in 2010 from University of Tehran, Iran, the M.S. degree in systems, control and robotics in 2012 from Royal Institute of Technology, Sweden, and the Ph.D. degree in systems and control in 2017 from the University of Groningen, The Netherlands. 
	He was a postdoctoral research associate with the Departments of Mathematical and Statistical Sciences and Computing Science at the University of Alberta from August 2017 to November 2020.
\end{IEEEbiography}
\begin{IEEEbiography}[]{Mohammad Hossein Roohi} is currently a postdoctoral research associate with the Departments of Electrical and Computer Engineering at the University of Alberta.
He received the B.Sc. and M.Sc. degrees in control systems from Isfahan University of Technology, Isfahan, Iran, in 2013 and 2016, respectively, and the Ph.D. degree in control systems in 2021 from the University of Alberta, Canada. 
\end{IEEEbiography}

\end{document}